\documentclass[a4paper,11pt, twoside]{article}
\usepackage{a4}
\usepackage{geometry} 
\usepackage{amsmath}
\usepackage{amssymb}
\usepackage{amsthm}
\usepackage{graphicx}
\usepackage{caption,subcaption}
\usepackage{multirow}
\usepackage{xstring}
\usepackage[utf8]{inputenc}
\usepackage{hyperref}
\usepackage{amstext,amsbsy,amsopn,eucal,enumerate}
\usepackage{tikz}
\usepackage{pgfplots}
\usepackage{tabularx}

\usepackage{soul}
\usepackage{algorithm} 
\usepackage{algpseudocode}
\def\R{\mathbb{R}}

\newcommand{\expn}{\operatorname{e}}
\newcommand{\opt}{\operatorname{opt}}

\newcommand{\diag}{\operatorname{diag}}

\newcommand{\kernel}{\operatorname{ker}}
\newcommand{\orth}{\operatorname{orth}}
\newcommand{\vect}{\operatorname{vec}}
\newcommand{\im}{\operatorname{im}}

\DeclareMathOperator*{\argmin}{arg\,min}
\DeclareMathOperator*{\markov}{markov}

\newcommand{\beq}{\begin{equation}}
\newcommand{\eeq}{\end{equation}}
\newcommand {\mat}      [1] {\left[\begin{array}{#1}}
\newcommand {\rix}          {\end{array}\right]}
\newcommand {\smat}      [1] {\left[\begin{smallmatrix}{#1}}
\newcommand {\srix}          {\end{smallmatrix}\right]}
\newcommand {\s}      [1] {\begin{smallmatrix}{#1}}
\newcommand {\se}          {\end{smallmatrix}}
\newcommand{\trace}{\operatorname{tr}}

\newtheorem{defn}{Definition}[section]
\newtheorem{remark}{Remark}

\newtheorem{lem}[defn]{Lemma}

\newtheorem{thm}[defn]{Theorem}

\usetikzlibrary{external}
\def\addlegendimage{\csname pgfplots@addlegendimage\endcsname}

\newlength\fheight
\newlength\fwidth

\title{Solving high-dimensional optimal stopping problems using optimization based model order reduction}

\author{Martin Redmann\thanks{Martin Luther University Halle-Wittenberg, Institute of Mathematics, Theodor-Lieser-Str. 5, 06120 Halle (Saale), Germany, Email: {\tt 
martin.redmann@mathematik.uni-halle.de}.}
}

\begin{document}

\maketitle

\begin{abstract}
Solving optimal stopping problems by backward induction in high dimensions is often very complex since the computation of conditional expectations is required. Typically, such computations are based on regression, a method that suffers from the curse of dimensionality. Therefore, the objective of this paper is to establish dimension reduction schemes for large-scale asset price models and to solve related optimal stopping problems (e.g. Bermudan option pricing) in the reduced setting, where regression is feasible. The proposed algorithm is based on an error measure between linear stochastic differential equations. We establish optimality conditions for this error measure with respect to the reduce system coefficients and propose a particular method that satisfies these conditions up to a small deviation. We illustrate the benefit of our approach in several numerical experiments, in which Bermudan option prices are determined.
\end{abstract}



\section{Introduction}

\subsection{Setting and reduced order approximation}
We consider the following linear stochastic system \begin{subequations}\label{original_system}
\begin{align}\label{stochstate}
 dx(t) &= Ax(t)dt + \sum_{i=1}^{q} N_i x(t-) dM_i(t),\quad x(0) = x_0 = X_0 z_0,\\ \label{stochout}
y(t) &= Cx(t),\quad t\in [0, T],
\end{align}
            \end{subequations}
where $x(t-):=\lim_{s\uparrow t} x(s)$. Above, we assume that $A, N_i \in \R^{n\times n}$ and $C \in \R^{p\times n}$. The columns of $X_0 \in \R^{n\times m}$ span the initial states of interest, where $z_0\in \R^m$ is a generic vector of coefficients associated to the expansion of a particular $x_0\in \im[X_0]$. Here, $\im[\cdot]$ denotes the image of a matrix. Let 
$M=\left(M_1, \ldots, M_{q}\right)^\top$ be an $\mathbb R^{q}$-valued square integrable L\'evy process with mean zero and covariance matrix $K_M=(k_{ij})_{i, j = 1, \ldots, q}$, i.e.,
$\mathbb E[M(t)M(t)^\top]= K_M t$ for $t\in [0, T]$. Such a matrix exists, see, e.g., \cite{zabczyk}.
$M$ and all stochastic process appearing in this paper are defined on a filtered probability space
$\left(\Omega, \mathcal F, (\mathcal F_t)_{t\in [0, T]}, \mathbb P\right)$\footnote{$(\mathcal F_t)_{t\in [0, T]}$ is right continuous and complete.}. In addition, $M$ is $(\mathcal F_t)_{t\in [0, T]}$-adapted and its 
increments $M(t+h)-M(t)$ are independent of $\mathcal F_t$ for $t, h\geq 0$ and $t+h\leq T$. 

System \eqref{original_system} can represent an $n$-dimensional asset price model, e.g., having the form  $
 A= \mathbf{\mathrm{r}} I$ and $N_i = \xi_i  e_i e_i^\top$,                                                                                                                                                                                                                                                                                                                                                                                
where $\mathbf{\mathrm{r}}\in\mathbb R$ is an interest rate, $\xi_i>0$ is a volatility parameter, $e_i$ is the $i$th unit vector in $\mathbb R^n$ and $q=n$. The quantity of interest then typically is a basket, i.e., $C=\frac{1}{n}\begin{bmatrix} 1 & 1 & \ldots & 1 \end{bmatrix}$ or a full state approximation is required as for instance in the context of max call options, where we set $C=I$. 
In this paper, the main focus will be on the Black-Scholes setting \eqref{original_system}. However, we will deliver strategies to extend this work to Heston-type models in Section \ref{sec_heston}. There, the constant covariance matrix $K_M$ will be replaced by a stochastic process. Now, the goal is to approximate \eqref{original_system} by a system 
\begin{subequations}\label{red_system}
\begin{align}\label{red_stochstate}
 d\hat x(t) &= \hat A\hat x(t) dt + \sum_{i=1}^{q} {\hat N}_i \hat x(t-) dM_i(t),\quad \hat x(0) = {\hat x}_0 = \hat X_0 z_0,\\ \label{red_stochout}
\hat y(t) &= \hat C\hat x(t),\quad t\in [0, T],
\end{align}
\end{subequations}
with a potentially  much smaller state dimension $\hat n\ll n$ and $y\approx \hat y$ in some sense. We have that $\hat x(t)\in\mathbb R^{\hat n}$  and $\hat A$, ${\hat N}_i \in \R^{\hat n\times \hat n}$ ($i=1,\ldots, q$), ${\hat X_0}\in \R^{\hat n\times m}$ and $\hat C\in \R^{p\times \hat n}$. 
Let us briefly sketch the idea how to derive \eqref{red_system} and the particular structure of the reduced system. The goal is to find an $\hat n$-dimensional subspace $\im[V]\subset \R^n$ that approximates the manifold of the state variable $x$, where $V\in\mathbb R^{n\times \hat n}$ is a full-rank matrix. Then, there exist a process $\hat x$ such that $V \hat x(t) \approx x(t)$. Given that this approximation is accurate, $\im[V]$ is called dominant subspace of \eqref{stochstate} or of $x$. Inserting the above estimate
into \eqref{original_system}, we have
\begin{align}\label{residual_eq}
 V\hat x(t) = X_0 z + \int_0^t AV \hat x(s) ds + \sum_{i=1}^{q} \int_0^t {N}_i V\hat x(s-) dM_i(s) + e(t)
\end{align}
with $y(t)\approx \hat y(t):= C V\hat x(t)$ and where $e(t)$ is the error in the state equation.
We assume/enforce the residual $e(t)$  to be orthogonal to a carefully chose second subspace $\im[W]$, where $W\in\mathbb R^{n\times \hat n}$ has full rank. Here, $\im[W]$ can, e.g., be the dominant subspace of the dual state equation ($A, N_i, X_0$ are replaced by $A^\top, N_i^\top, C^\top$ in \eqref{stochstate}). Now, multiplying \eqref{residual_eq} with $(W^\top V)^{-1} W^\top$ from the left leads to \eqref{red_system} 
with coefficients\begin{align}\label{red_mat_projection}
  \hat A = (W^\top V)^{-1}W^\top A V , \quad \hat N_i = (W^\top V)^{-1}W^\top N_i V , \quad    \hat X_0 = (W^\top V)^{-1}W^\top X_0,\quad \hat C = C V.
     \end{align}
The associated state approximation $V \hat x$ can be viewed as a process resulting from a Petrov-Galerkin method using the (non-orthogonal) projection $V(W^\top V)^{-1}W^\top$ onto $\im[V]$. For that reason, we also call $V$ and $W$ projection matrices. These projection matrices can be derived using different strategies. A very popular methods for deterministic control systems is balanced truncation \cite{moore} which has been studied for stochastic control systems as well \cite{beckerhartredrich, bennerdamm, redmannbenner}. Here, we have $W^\top V = I$. Extending such concepts to asset price models is non trivial as they have different properties than the previously studied equations. A first work exploiting such techniques in the option pricing context can be found in \cite{mor_heston}. Another important class is Proper Orthogonal Decomposition \cite{pod} that in contrast to the previously mentioned scheme relies on an orthogonal projection, i.e., $V = W$ is an orthogonal matrix. This type of methods has been applied to stochastic differential equations \cite{pod_sde} and in the finance context as well \cite{Bermudan_POD}. A third class are Krylov subspaces/optimization based model reduction techniques such as the Iterative Rational Krylov Algorithm (IRKA) that has been widely investigated for deterministic controlled equations \cite{linIRKA} but also exist for stochastic systems \cite{mliopt}. The key idea is to find an upper bound for the error between two systems like \eqref{original_system} and \eqref{red_system}, establish associated optimality conditions w.r.t. the reduced order coefficients and subsequently construct a model that is optimal in terms of the underlying error measure. The goal of this paper is to create such IRKA-type schemes for general uncontrolled linear stochastic differential equations like asset price models that, in contrast to \cite{mliopt}, do not rely on restrictive assumptions such as asymptotic stability or zero initial conditions. Comparing our approach with other dimension reduction techniques like \cite{mor_heston}, the advantage is that it usually is computationally less expensive and applicable in higher dimensions. We refer to  \cite{morbook1, morbook2} for a general and comprehensive overview on model reduction techniques for deterministic (control) systems.
\smallskip

Below, we point out why low-order approximations are meaningful in the context of optimal stopping problems.

\subsection{Motivation, objective and outline of this paper}\label{intro2}

Our goal is to solve (Markovian) optimal stopping problems associated to \eqref{original_system} as they for instance occur when pricing American/Bermudan options. Due to the enormous computational complexity for large $n$, we aim to use the low-dimensional approximation in \eqref{red_system} to find an estimate for the optimal value of the stopping problem: 
 \begin{align}\label{approx_stop}
  \sup_{\tau \in \mathcal{S}_0}\mathbb E \left[f_\tau(y(\tau))\right] \approx \sup_{\tau \in \hat{\mathcal{S}}_0}\mathbb E \left[f_\tau(\hat y(\tau))\right],                                                                                                                                               \end{align}
where $\mathcal S_0$ and $\hat{\mathcal{S}}_0$ are sets of stopping times associated to the filtrations generated by the price processes $x$ and $\hat x$, respectively, and restricted to a set of exercise dates $\mathcal J\subset [0, T]$. Moreover, in the option pricing context,  $f$ is a discounted payoff function. The 
benefit of the approximation in \eqref{approx_stop} is that backward dynamic programming might be applicable in the reduced setting \eqref{red_system}. Dynamic programming is often not feasible in the original framework \eqref{original_system} since it requires to compute continuation functions of the form \begin{align}  \label{continuation_function}                                                                                                                                                              
c_s(x)=\mathbb E[V_{t}(x({t})) | x(s)=x],\quad x\in\mathbb R^n, \quad s< t,  \quad s, t\in\mathcal J,
\end{align}
where $V_t$ is the value function at time $t$ and $c_s(x)$ describes the expected future value of the option given the state of the underlyings at time $s$. A Bermudan option is now exercised at $s$ if the current payoff is the first time larger or equal than the continuation value at $s$. Given a suitable (polynomial) basis $\psi_{1}, \ldots, \psi_{\mathfrak K}$, we can find an approximation $c_s(\cdot) \approx\sum_{k=1}^{\mathfrak K} {\beta}_k^{\opt}\psi_{k}(\cdot)$ from the least squares problem
\begin{equation}\label{eq:least-squares-regression} 
{\beta}^{\opt} := \argmin_{\beta
\in\mathbb{R}^{\mathfrak K}} \sum_{i=1}^{\mathfrak M} \left|  V_t(x(t)^{i}) - \sum_{k=1}^{\mathfrak K} \beta_{k}
\psi_{k}(x(s)^{i}) \right|  ^{2},
\end{equation}
where $x(t)^{i}$ and $x(s)^{i}$ ($i=1, \dots, \mathfrak M$) are i.i.d.~samples of the random variables $x(t)$ and $x(s)$, respectively. We refer to \cite{LS2001, TV2001} for such regression based schemes used for pricing options. Now, solving the least squares problem in \eqref{eq:least-squares-regression} goes along with a large computational burden already in moderate high dimensions since this method suffers from the curse of dimensionality and is hence often not applicable for $n \geq 10$. Therefore, the continuation functions shall be computed in the reduced setting based on \eqref{eq:least-squares-regression}.   However, a sufficiently large reduction potential of the problems is required to be able to choose $\hat n$ small enough while the approximation error is low, e.g., to be able to apply the algorithm of Longstaff and Schwartz \cite{LS2001}.\smallskip

The paper, is now organized as follows. In Section \ref{sec_error_measure}, an error bound between two general systems of the form \eqref{original_system} and \eqref{red_system} is derived. The hope is that finding a candidate for \eqref{red_system} ensuring a small/minimal bound leads to a good approximation. Based on a Gronwall lemma for matrix differential equations, it is further shown that the same bound applies to a certain class of asset price models with stochastic volatilities allowing to directly transfer the results of this paper from the Black-Scholes to this more general setting.
For the error measure of Section \ref{sec_error_measure}, necessary optimality conditions for local minimality are proved in Section \ref{section3} which is one of our main results.  Here, very different techniques in comparison to similar approaches in other settings are required due to the higher complexity of the error measure. Moreover, we construct a particular model reduction scheme designed to approximately fit the optimality conditions associated to the error bound. The relation between the gap in these conditions and the covariance error at time $T$ is pointed out. In addition, a link between the error in the optimality conditions and eigenvalues of system covariances as well as certain singular values of \eqref{original_system} is shown. Such values provide algebraic criteria for a potentially low-dimensional underlying structure. Section \ref{mor_stab_sys} deals with the asymptotic behavior of the proposed model reduction algorithm. In fact, the error in the optimality conditions vanishes as $T\rightarrow \infty$ in case mean square asymptotic stability is assumed for \eqref{original_system} and \eqref{red_system}. The paper is concluded by several numerical experiments in Section \ref{sec_num} showing the benefit of our dimension reduction algorithm. In particular, fair prices of Bermudan basket and max call options are computed in the reduced setting representing highly accurate approximations of the original prices.

\subsection{Literature review on alternative methods}
Apart from our dimension reduction approach, there are several other concepts for pricing options given high-dimensional underlyings. One way is to exploit that value functions associated to options are solutions of partial differential equations (PDEs), e.g., Black-Scholes PDEs for European or Hamilton-Jacobi-Bellman equations for American options. In such PDE settings, the spatial variable is of dimension $n$ which does not allow for traditional computational tools like finite element and finite difference discretizations as they suffer from the curse of dimensionality. Many recent approaches for solving high-dimensional problems are based on machine learning, in particular deep neural
networks, see, e.g., \cite{machinelearning1,machinelearning2}. These methods often offer effective computational approaches which exploit an underlying low effective dimensionality. Explicitly detecting these low-dimensional structures is the goal of the model reduction method that we investigate in this paper.\smallskip

Another way to solve high-dimensional PDEs is the use of sparse grid approximations. Roughly speaking, 
a 1D grid of size $\mathcal N$ is turned into a ``tensor-product'' grid of size $\mathcal N^n$ in dimension $n$. This
explosion can often be avoided by a careful choice of a sparse ``subgrid''. Under suitable regularity conditions,
similar accuracy can be achieved with sparse grids of size asymptotically proportional to $\mathcal N \log(\mathcal N)^{n-1}$. Therefore, sparse girds can be viewed as an optimization procedure on a discrete level, whereas model reduction optimizes before any discretization. We refer to \cite{bungartzgriebel04} for a general exposition of sparse grid methods and to \cite{sparse_grid_habil, sparse_fin} for applications of sparse grids in the option pricing context.\smallskip

There are many works tackling high-dimensional stopping (or option pricing) problems more directly. Once again, neural networks an be exploited to learn the optimal stopping strategy $\tau^*\in \mathcal{S}_0$ in \eqref{approx_stop}, see  \cite{learn_tau1, learn_tau2}, or to approximate the continuation functions in \eqref{continuation_function}, see \cite{learn_continuation1, learn_continuation2}.\smallskip

Alternatively, if $p$ is small, one can think of doing a regression in $y$ instead of conducting it in $x$ to compute the continuation functions in \eqref{continuation_function}. However, this would lead to a non Markovian setting. Therefore, Markovian projections \cite{brunick2013mimicking, gyongy1986mimicking} are of interest in which a Markov process $y_{\markov}$ is found, so that $y_{\markov}(t)$ and $y(t)$ have the same distribution for each $t\in [0, T]$. Consequently, we have a weak but exact approximation of our quantity of interest. This is in contrast to our approach, where a good $L^2$-estimate $\hat y$ of $y$ shall be found. Markovian projections have been successfully applied to price options \cite{bayer2019implied, hambly2016forward}. \smallskip

Last, we want to refer to a very different approach aiming to lower the complexity of regression by tensor methods \cite{bermudan_tensor}. The high complexity of regression comes from the number of required coefficients $\beta_k$ in \eqref{eq:least-squares-regression} that is exponentially increasing in $n$. In \cite{bermudan_tensor}, the key idea is to reduce the degrees of freedom in the regression by using tensorized polynomial expansions.

\section{Covariance functions, error measures and dual systems}\label{sec_error_measure}

\subsection{Black-Scholes models}\label{blabla}

\paragraph{Covariance functions and error measures.}
We introduce the $\R^{n\times n}$-valued  fundamental solution $\Phi$ associated to \eqref{stochstate} as the solution to
\begin{align}\label{funddef}
 \Phi(t)=I+\int_0^t A \Phi(s) ds+\sum_{i=1}^{q} \int_0^t N_i \Phi(s-)dM_i(s), \quad t\in [0, T].
\end{align}
It immediately follows that we obtain a representation of the state by $x(t) =  \Phi(t) x_0 =  \Phi(t) X_0 z_0$. In addition, the fundamental solution to \eqref{red_stochstate} shall be denoted by $\hat \Phi$. The covariance function $F(t):= \mathbb E \left[\Phi(t) X_0 X_0^\top \Phi(t)^\top\right]$ will play an essential role when analyzing the error between \eqref{original_system} and \eqref{red_system}. This function $F$ can be calculated via a matrix differential equation. In this context, let us introduce the Lyapunov operator $\mathcal L$ by \begin{align}\label{def_lyap_op} \mathcal L(X) = A X+ X {A}^\top+ \sum_{i, j = 1}^{q} N_i X {N}_j^\top\;k_{ij}\end{align}
and its adjoint with respect to the Frobenius inner product $\langle \cdot, \cdot \rangle_F$ which is $\mathcal L^*(X) = A^\top X+ X {A}+\sum_{i, j = 1}^{q} N_i^\top X {N}_j\;k_{ij}$. Throughout this paper, $\mathcal L$ is assumed to be invertible.  
\begin{lem}\label{lemdgl}
Let $\Phi$ be the fundamental solution of \eqref{stochstate}. Then,  
the $\mathbb R^{n\times n}$-valued function $\mathbb E \left[\Phi(t) X_0 X_0^\top \Phi(t)^\top\right]$, $t\in [0, T]$, satisfies \begin{align}\label{matrixdgl}
\dot{F}(t)=\mathcal L[F(t)],\quad F(0)= X_0 X_0^\top.
\end{align}
Using the vectorization $\vect[\cdot]$ of a matrix, \eqref{matrixdgl} is equivalent to 
\begin{align}\label{vect_matrixdgl}
\frac{d}{dt}{\vect[F(t)]}=\mathcal K \vect[F(t)],\quad \vect[F(0)]= \vect[X_0 X_0^\top],
\end{align}
where \begin{align*}                                                                                                \mathcal K:= I\otimes A + A\otimes I+ \sum_{i, j=1}^q N_i \otimes N_j k_{ij}.                                                                                                                        \end{align*}
\end{lem}
\begin{proof}
A more general version of the result in \eqref{matrixdgl} can be found in \cite[Lemma 2.1]{mliopt}. Applying the linear and invertible $\vect$-operator to both sides of \eqref{matrixdgl} and exploiting that $\vect(A_1 X A_2)=  (A_2^\top \otimes A_1) \vect(X)$ for matrices $A_1, A_2, X$ of suitable dimension, \eqref{vect_matrixdgl} immediately follows.
\end{proof}
Lemma \ref{lemdgl} is used below to prove an $L^2$-error bound between \eqref{original_system} and \eqref{red_system} that is the basis for the later algorithm leading to a suitable candidate for \eqref{red_system}. In order to derive this bound, we can combine the state equations in \eqref{stochstate} and \eqref{red_stochstate} to obtain a system
{\allowdisplaybreaks \begin{align}\label{errorsys:original}
d \smat x(t) \\ \hat x(t) \srix &= \smat{A}& 0\\ 
0 &{\hat A}\srix \smat x(t) \\ \hat x(t) \srix dt + \sum_{i=1}^q \smat {N}_i& 0 \\ 
0 &\hat N_i\srix \smat x(t-) \\ \hat x(t-) \srix dM_i(t),\quad \smat x(0) \\ \hat x(0) \srix = \smat X_0 \\  \hat X_0\srix z_0,
\end{align}}
with fundamental solution $\smat \Phi(t)& 0 \\ 0& \hat \Phi(t) \srix$, $t\in [0, T]$. The error between the output $y$ and $\hat y$ given by \eqref{stochout} and \eqref{red_stochout} can be expressed using the state variable of \eqref{errorsys:original}, since
\begin{align}\nonumber
&\mathbb E\int_0^T \left\|y(t)-\hat y(t) \right\|_2^2dt = \mathbb E\int_0^T \left\| \smat C & -\hat C \srix \smat x(t) \\ \hat x(t) \srix \right\|_2^2dt \\ \nonumber
&= \mathbb E\int_0^T \left\| \smat C & -\hat C \srix \smat \Phi(t)& 0 \\ 0& \hat \Phi(t) \srix  \smat X_0 \\  \hat X_0\srix z_0\right\|_2^2dt
\leq \mathbb E\int_0^T \left\| \smat C & -\hat C \srix \smat \Phi(t)& 0 \\ 0& \hat \Phi(t) \srix  \smat X_0 \\  \hat X_0\srix \right\|_F^2dt \left\| z_0\right\|_2^2\\
&= \mathbb E\int_0^T\trace\left(\smat C & -\hat C \srix \smat \Phi(t)& 0 \\ 0& \hat \Phi(t) \srix  \smat X_0 \\ \label{first_step_bound} \hat X_0\srix \smat X_0 \\  \hat X_0\srix^\top \smat \Phi(t)& 0 \\ 0& \hat \Phi(t) \srix^\top \smat C & -\hat C \srix^\top  \right) dt \left\| z_0\right\|_2^2,
 \end{align}
where $\left\| \cdot\right\|_2$ denotes the Euclidean norm. We apply Lemma \ref{lemdgl} to equation \eqref{errorsys:original} and obtain that the error covariance $F^{err}(t):= \mathbb E\left[\smat \Phi(t)& 0 \\ 0& \hat \Phi(t) \srix  \smat X_0 \\  \hat X_0\srix \smat X_0 \\  \hat X_0\srix^\top \smat \Phi(t)& 0 \\ 0& \hat \Phi(t) \srix^\top\right]$ satisfies \begin{align*} \frac{d}{dt}{F^{err}}(t) = \smat{A}& 0\\ 
0 &{\hat A}\srix F^{err}(t)+ F^{err}(t) \smat{A}& 0\\ 
0 &{\hat A}\srix^\top+ \sum_{i, j = 1}^{q} \smat {N}_i& 0 \\ 
0 &\hat N_i\srix F^{err}(t) \smat {N}_j& 0 \\ 
0 &\hat N_j\srix^\top\;k_{ij}\end{align*}
with $F^{err}(0)= \smat X_0 \\  \hat X_0\srix \smat X_0 \\  \hat X_0\srix^\top$.
 Consequently, the right lower block $\hat F(t):= \mathbb E \left[\hat \Phi(t) \hat X_0 \hat X_0^\top \hat \Phi(t)^\top\right]$ of $F^{err}(t)$ is the solution to 
 \begin{align}
\label{matrixequalforFred}
\dot {\hat F}(t) &=\hat {\mathcal L}[F(t)],\quad \hat F(0)=\hat X_0 \hat X_0^\top,\\ \nonumber
\bigg(\Leftrightarrow \frac{d}{dt}{\vect[\hat F(t)]}&=\hat{\mathcal K} \vect[\hat F(t)],\quad \vect[\hat F(0)]= \vect[\hat X_0 \hat X_0^\top]\bigg),
\end{align}
with $\hat {\mathcal L}$ denoting the reduced system Lyapunov operator, where the original matrices are replace by the reduced coefficients in \eqref{def_lyap_op}. $\hat {\mathcal L}$ is also supposed to be invertible throughout this paper. The corresponding Kronecker representation is $\hat {\mathcal K}:= I\otimes \hat A + \hat A\otimes I+ \sum_{i, j=1}^q \hat N_i \otimes \hat N_j k_{ij}$. Moreover, the right upper block $\tilde F(t):= \mathbb E \left[ \Phi(t) X_0 \hat X_0^\top \hat \Phi(t)^\top\right]$ of $F^{err}(t)$ satisfies
\begin{align}\label{matrixequalforFmixed}
\dot {\tilde F}(t) &= \tilde {\mathcal L}[\tilde F(t)],\quad \tilde F(0)=X_0 \hat X_0^\top,\\ \nonumber
\bigg(\Leftrightarrow \frac{d}{dt}{\vect[\tilde F(t)]}&=\tilde{\mathcal K} \vect[\tilde F(t)],\quad \vect[\tilde F(0)]= \vect[X_0 \hat X_0^\top]\bigg)
    \end{align}
with $\tilde {\mathcal L}(\tilde X):=  A\tilde X+\tilde X\hat A^\top+ \sum_{i, j=1}^q N_i \tilde  X \hat N_j^\top k_{ij}$
being the mixed operator containing both reduced and original system matrices. In addition, the associated Kronecker matrix is  $\tilde {\mathcal K}:= I\otimes A + \hat A\otimes I+ \sum_{i, j=1}^q \hat N_i \otimes N_j k_{ij}$. Multiplying out the matrix products in \eqref{first_step_bound}, 
we obtain the desired error measure that we formulate in the following lemma.
\begin{lem}\label{lemma_error_measure}
Let $y$ be the quantity of interest given in \eqref{stochout} and $\hat y$ in \eqref{red_stochout} its potential approximation. Then, we have
 \begin{align}\label{error_measure_bound}
\mathbb E\int_0^T \left\|y(t)-\hat y(t) \right\|_2^2dt \leq \left(\trace(C P(T) C^\top)-2 \trace(C \tilde P(T) \hat C^\top)+\trace(\hat C \hat P(T) \hat C^\top) \right)\left\|z_0\right\|_2^2, \end{align}
where $P(T):=\int_0^T F(t)dt$, $\hat P(T):=\int_0^T \hat F(t)dt$, $\tilde P(T):=\int_0^T \tilde F(t)dt$  are the integrals/time-averages of the covariance functions $F, \tilde F$ and $\hat F$ given by the equations \eqref{matrixdgl},   \eqref{matrixequalforFred} and \eqref{matrixequalforFmixed}.
\end{lem}
The aim is to derive a system \eqref{red_system} with a small bound in \eqref{error_measure_bound} leading to an accurate approximation of $y$ by $\hat y$. In order to achieve this goal necessary conditions for local optimality shall be determined and subsequently a reduced model is constructed that has coefficients being very close to a local minimum of the right-hand side of \eqref{error_measure_bound}. Notice that an accurate $L^2$-estimate is beneficial recalling that conditional expectations like in \eqref{continuation_function} need to be approximated. Such conditional expectations can be interpreted as  orthogonal projections in $L^2(\Omega, \mathcal F, \mathbb P)$.\smallskip

Before we compute the above mentioned necessary optimality conditions for the error bound in \eqref{error_measure_bound}, the dual system of \eqref{stochstate} and its covariance functions are discussed. Those play an important role in the formulation of the optimality conditions.

\paragraph{Dual state equations and their covariance functions.}
Talking about the dual state variable of \eqref{stochstate}, we mean the process $x_d$ satisfying \begin{align}\label{stochstate_dual}
 dx_d(t) &= A^\top x_d(t)dt + \sum_{i=1}^{q} N_i^\top x_d(t-) dM_i(t),\quad x_d(0) = x_{d, 0} = C^\top z_{d, 0}
\end{align}
with $z_{d, 0}\in \mathbb R^p$ being the coefficients of the expansion belonging to the generic initial state $x_{d, 0}$. Let $\Phi_d$ and $\hat \Phi_d$ denote the fundamental solutions to \eqref{stochstate_dual} and the dual reduced system, respectively. The dual covariances will play an essential role in the construction of the reduced system \eqref{red_system}. As in the previous paragraph, they are introduced as $G(t):= \mathbb E \left[\Phi_d(t) C^\top C \Phi_d(t)^\top\right]$, $\hat G(t):= \mathbb E \left[\hat \Phi_d(t) \hat C^\top \hat C \hat \Phi_d(t)^\top\right]$ and $\tilde G(t):= \mathbb E \left[\Phi_d(t) C^\top \hat C \hat \Phi_d(t)^\top\right]$. They solve
\begin{align}
\label{matrixequalforG}
\dot {G}(t) &= {\mathcal L}^*[G(t)],\quad G(0)= C^\top C,\\ \nonumber
\bigg(\Leftrightarrow \frac{d}{dt}{\vect[G(t)]}&={\mathcal K}^\top \vect[ G(t)],\quad \vect[G(0)]= \vect[C^\top C]\bigg),\\
\label{matrixequalforGred}
\dot {\hat G}(t) &= \hat {\mathcal L}^*[\hat G(t)],\quad \hat G(0)=\hat C^\top \hat C,\\ \nonumber
\bigg(\Leftrightarrow \frac{d}{dt}{\vect[\hat G(t)]}&=\hat{\mathcal K}^\top \vect[\hat G(t)],\quad \vect[\hat G(0)]= \vect[\hat C^\top \hat C]\bigg),\\
\label{matrixequalforGmixed}
\dot {\tilde G}(t) &= \tilde {\mathcal L}^*[\tilde G(t)],\quad \tilde G(0)=C^\top \hat C,\\ \nonumber
\bigg(\Leftrightarrow \frac{d}{dt}{\vect[\tilde G(t)]}&=\tilde{\mathcal K}^\top \vect[\tilde G(t)],\quad \vect[\tilde G(0)]= \vect[\hat C^\top C]\bigg).
    \end{align}
We introduce the corresponding time averaged covariance functions by $ Q(T):=\int_0^T G(t)dt$, $\hat Q(T):=\int_0^T \hat G(t)dt$ and $\tilde Q(T):=\int_0^T \tilde G(t)dt$.
  
\subsection{Extension to Heston-type models}\label{sec_heston}

In this section, we briefly discuss possible extensions of our results to asset price models with stochastic volatility. A famous approach in this direction was given by Heston \cite{heston} using a scalar Cox–Ingersoll–Ross (CIR) process \cite{CIR}. There is many different extensions to matrix-valued volatility processes such as in \cite{wishart_vol}, where a Wishart processes \cite{wiashart} was involved. It is also possible to study volatilities with jumps \cite{jump_vol}. There,  matrix-valued Ornstein-Uhlenbeck type volatilities driven by a L\'evy process are exploited. We refer to \cite{Cuchiero} for a more detailed discussion. In this paper, we can also consider matrix-valued stochastic volatilities. However, we need them to be bounded in a certain sense. A possible extension of \eqref{original_system} can be \begin{subequations}\label{original_system_heston}
\begin{align}\label{heston_state}
 dx_H(t) &= Ax_H(t)dt + \begin{bmatrix} N_1 x_H(t)& N_2 x_H(t)& \dots
 &N_q x_H(t)\end{bmatrix} \mathbb K(t)^{\frac{1}{2}} dB(t),\\ 
y_H(t) &= Cx_H(t),\quad t\in [0, T],
\end{align}
\end{subequations}
where for simplicity $B$ is a standard Brownian motion and $\mathbb K$ is an $(\mathcal F_t)_{t\in [0, T]}$-adapted and $\mathbb R^{q\times q}$-valued process that satisfies \begin{align}\label{bounded_vol}
   \mathbb K(t) \leq K_M \quad \mathbb P-\text{a.s. and for all } t\in [0, T],                                                                                                                                                                                                                                                        \end{align}
with $K_M$ being a positive semidefinite matrix and with ``$\leq$'' being meant in terms of definiteness. A more concrete choice can be $\mathbb K(t) = v(t) K$, where $K$ is a constant positive semidefinite matrix, $v(t) = \min\{v_{cir}(t), c\}$, $v_{cir}$ is a scalar  CIR process and $c>0$ is a constant. This latter setting was, e.g., investigated in \cite{mor_heston}. Instead of a scalar CIR, one might also think of involving multiple one, e.g., by setting $\mathbb K(t) = \diag(v_{1}(t),\ldots, v_{q}(t))$, where $v_j(t) = \min\{v_{j, cir}(t), c_j\}$ are of the form $v$ is given above. The following lemma shows why \eqref{bounded_vol} is required.
\begin{lem}\label{lemineqdgl}
Let $\Phi_H$ be the fundamental solution of \eqref{heston_state} with $\mathbb K$ as in \eqref{bounded_vol}. Then, 
$F_H(t):=\mathbb E \left[\Phi_H(t) X_0 X_0^\top \Phi_H(t)^\top\right]$, $t\in [0, T]$, satisfies the matrix inequality \begin{align}\label{matrix_ineq}
\dot{F}_H(t)\leq \mathcal L[F_H(t)],\quad F_H(0)= X_0 X_0^\top.
\end{align}
Moreover, we have $F_H(t)\leq F(t)= \mathbb E \left[\Phi(t) X_0 X_0^\top \Phi(t)^\top\right]$, $t\in [0, T]$, where $F$ solves \eqref{matrixdgl}.
\end{lem}
\begin{proof}
We can write $F_H(t)=\sum_{i=1}^m\mathbb E \left[\Phi_H(t) X_0 e_i e_i^\top X_0^\top \Phi_H(t)^\top\right]$, where $e_i$ is the $i$th unit vector in $\mathbb R^m$. Now, $\Phi_H(\cdot) X_0 e_i$ is the solution $x_H$ of \eqref{heston_state} with initial state $x_0= X_0 e_i$. For that reason, it remains to show that inequality \eqref{matrix_ineq} holds for the function $\mathbb E \left[x_H(t) x_H(t)^\top\right]$, $t\in [0, T]$, using a generic (deterministic) initial condition $x_0$. We apply Ito's product rule (see, e.g., \cite{Oksendal}) and obtain \begin{align}\nonumber
&d\big(x_H(t) x_H(t)^\top\big)=  \big(dx_H(t)\big) x_H(t)^\top + x_H(t) d\big(x_H(t)^\top\big) \\ \nonumber
&+ \begin{bmatrix} N_1 x_H(t)& N_2 x_H(t)& \dots
 &N_q x_H(t)\end{bmatrix} \mathbb K(t) \begin{bmatrix} N_1 x_H(t)& N_2 x_H(t)& \dots
 &N_q x_H(t)\end{bmatrix}^\top dt\\ \label{heston_cov}
 &\leq \big(dx_H(t)\big) x_H(t)^\top + x_H(t) d\big(x_H(t)^\top\big)+\sum_{i, j = 1}^{q} N_i x_H(t) x_H(t)^\top {N}_j^\top\;k_{ij} dt
\end{align}
exploiting \eqref{bounded_vol} and rewriting the matrix multiplication after replacing $\mathbb K(t)$ by $K_M=(k_{ij})$. With \eqref{heston_state} and mean zero property of the Ito integral (see \cite{Oksendal}), we obtain $\mathbb E\left[\big(dx_H(t)\big) x_H(t)^\top\right] = A \,\mathbb E\left[x_H(t) x_H(t)^\top\right] dt$. Inserting this into \eqref{heston_cov} after the expectation has been applied leads to $\frac{d}{dt} \mathbb E \left[x_H(t) x_H(t)^\top\right]\leq\mathcal L\left[\mathbb E \left[x_H(t) x_H(t)^\top\right]\right]$. This concludes the first part of the proof. Now, we exploit the Gronwall result of \cite[Lemma 2.3]{mor_heston} that tells that the solution of \eqref{matrixdgl} dominates all functions satisfying \eqref{matrix_ineq} in terms of definiteness if $\mathcal L$ is a resolvent positive operator. This property holds true for the Lyapunov operators considered here. This yields the claim.
\end{proof}
Now, the reduced model candidate for \eqref{original_system_heston} is \begin{subequations}\label{original_system_heston_reduced}
\begin{align}
 d\hat x_H(t) &= \hat A \hat x_H(t)dt + \begin{bmatrix} \hat N_1 \hat x_H(t)& \hat N_2 \hat x_H(t)& \dots
 &\hat N_q \hat x_H(t)\end{bmatrix} \mathbb K(t)^{\frac{1}{2}} dB(t),\\ 
\hat y_H(t) &= \hat C \hat x_H(t),\quad t\in [0, T].
\end{align}
\end{subequations}
with $\hat x_H(t)\in \mathbb R^{\hat n}$ and fundamental solution $\hat \Phi_H$. We can define an error system for \eqref{original_system_heston} and \eqref{original_system_heston_reduced}, as done in \eqref{errorsys:original}, that has a fundamental solution $\smat \Phi_H(t)& 0 \\ 0& \hat \Phi_H(t) \srix$. Following the steps in \eqref{first_step_bound}, we obtain \begin{align}\label{heston_error_first}
&\mathbb E\int_0^T \left\|y_H(t)-\hat y_H(t) \right\|_2^2dt \\ \nonumber
&\leq \mathbb E\int_0^T\trace\left(\smat C & -\hat C \srix \smat \Phi_H(t)& 0 \\ 0& \hat \Phi_H(t) \srix  \smat X_0 \\  \hat X_0\srix \smat X_0 \\  \hat X_0\srix^\top \smat \Phi_H(t)& 0 \\ 0& \hat \Phi_H(t) \srix^\top \smat C & -\hat C \srix^\top  \right) dt \left\| z_0\right\|_2^2.
 \end{align}
Using Lemma \ref{lemineqdgl}, it holds that \begin{align*}                                        
\mathbb E \left[\smat \Phi_H(t)& 0 \\ 0& \hat \Phi_H(t) \srix  \smat X_0 \\  \hat X_0\srix \smat X_0 \\  \hat X_0\srix^\top \smat \Phi_H(t)& 0 \\ 0& \hat \Phi_H(t) \srix^\top\right]\leq \mathbb E \left[\smat \Phi(t)& 0 \\ 0& \hat \Phi(t) \srix  \smat X_0 \\  \hat X_0\srix \smat X_0 \\  \hat X_0\srix^\top \smat \Phi(t)& 0 \\ 0& \hat \Phi(t) \srix^\top\right],
\end{align*}
i.e., the error covariance of the Black-Scholes model dominates the one of the Heston-type model. Applying this to \eqref{heston_error_first}, we find \begin{align*}
\mathbb E\int_0^T \left\|y_H(t)-\hat y_H(t) \right\|_2^2dt \leq \left(\trace(C P(T) C^\top)-2 \trace(C \tilde P(T) \hat C^\top)+\trace(\hat C \hat P(T) \hat C^\top) \right)\left\|z_0\right\|_2^2,                                                                                                                                                                                                                                                                                                       \end{align*}
where the bound is the same as in Lemma \ref{lemma_error_measure}. Consequently, the coefficients $\hat A, \hat N_i, \hat X_0, \hat C$ leading to a small bound in \eqref{error_measure_bound} also yield a good reduced system \eqref{original_system_heston_reduced} for \eqref{original_system_heston}. This means that all results of this paper can immediately transferred to models with stochastic but bounded volatility (in the sense of \eqref{bounded_vol}). 

\section{Reduced order models based on error bound minimization}\label{section3}
Below, we construct a reduced system \eqref{red_system} potentially having a small (almost locally minimal) bound \eqref{error_measure_bound}. This (hopefully) provides an accurate $L^2$-approximation for the output $y$ of \eqref{original_system} exploiting Lemma \ref{lemma_error_measure}. With this $L^2$-estimate, a high accuracy in \eqref{approx_stop} shall subsequently be achieved.
\subsection{Necessary optimality conditions for the $L^2$-error measure of Lemma \ref{lemma_error_measure} }
Based on the bound in \eqref{error_measure_bound}, we seek for necessary conditions for local optimality of the expression 
\begin{align}\label{calE}
\mathcal E(\hat A, \hat N_i, \hat X_0, \hat C) := \trace(\hat C \hat P(T) {\hat C}^\top) - 2 \trace(C \tilde P(T) {\hat C}^\top)= \langle {\hat C}^\top \hat C, \hat P(T)\rangle_F - 2 \langle  C^\top {\hat C}, \tilde P(T)\rangle_F.
\end{align}
The following theorem is one of our main results telling the criteria for a good reduced system.
\begin{thm}\label{thm_opt_cond}
Given a reduced model \eqref{red_system} that is locally optimal with respect to the bound of Lemma \ref{lemma_error_measure}. Then, 
for the matrices $\hat A, \hat N_i, \hat X_0, \hat C$, it holds that 
\begin{equation}\label{optcond}
\begin{aligned}
(a)\quad&\hat C \hat P(T) = C \tilde P({T}),\quad (b)\quad \hat Q(T)\hat X_0  =  \tilde Q({T})^\top X_0\\
(c)\quad & \int_0^T  \hat Q(T-t)\hat F(t) dt  =\int_0^T  \tilde Q(T-t)^\top \tilde F(t) dt, \\
(d)\quad& \int_0^T \hat Q(T-t) \left(\sum_{j=1}^{q}\hat N_j k_{ij} \right) \hat F(t) dt =  \int_0^T \tilde Q(T-t)^\top \left(\sum_{j=1}^{q} N_j k_{ij}\right) \tilde F(t) dt
\end{aligned}
\end{equation}
for $i=1,\dots, q$ and where all covariance matrices entering \eqref{optcond} are defined in Section \ref{blabla}.
\end{thm}
\begin{proof}
Since the proof of this theorem requires several long and technical calculations, it is moved to Appendix \ref{app_pending_proof} in order to improve the readability of this paper.
\end{proof}
The result of Theorem \ref{thm_opt_cond}  is important as it allows to reduce the problem of selecting suitable projection matrices $V$ and $W$ for \eqref{original_system} to identifying dominant subspaces of the matrix differential equations \eqref{matrixequalforFmixed} and \eqref{matrixequalforGmixed}. 
The following lemma shall emphasize this aspect and provide an intuition what is required to approximately satisfy the optimality conditions in Theorem \ref{thm_opt_cond} without aiming to be specific about the error terms. 
\begin{lem}\label{lem_cond_opt}
Given a reduced system \eqref{red_system} with coefficients of the form \eqref{red_mat_projection} and a small point-wise error in \begin{align}\label{almost_opt_cond}
V\hat F \approx \tilde F\quad\text{and}\quad W (V^\top W)^{-1} \hat G \approx \tilde G
             \end{align}
on $[0, T]$. Then, we have a small deviation in the optimality conditions in Theorem \ref{thm_opt_cond}. 
\end{lem}
\begin{proof}
Since \eqref{almost_opt_cond} holds point-wise, the approximation is also accurate for the associated integrals, i.e., we have a small error in \begin{align}     \label{small_int_erro}                                                                                                                                   
V\hat P(t) \approx \tilde P(t)\quad\text{and}\quad W (V^\top W)^{-1} \hat Q(t) \approx \tilde Q(t)                                                                                                                                             \end{align}
for all $t \in [0, T]$. Inserting \eqref{small_int_erro} for $t=T$ into the right-hand sides of \eqref{optcond} (a) and (b), we obtain 
$C \tilde P({T})\approx  C V\hat P(T) = \hat C \hat P(T)$ and $\tilde Q({T})^\top X_0 \approx \hat Q(T)  (W^\top V)^{-1} W^\top X_0 = \hat Q(T)\hat X_0 $ 
using the definitions of $\hat X_0$ and $\hat C$ given in \eqref{red_mat_projection}. Exploiting \eqref{almost_opt_cond} and \eqref{small_int_erro} for the right-hand sides of \eqref{optcond} (c) and (d) yields  $\int_0^T  \tilde Q(T-t)^\top \tilde F(t) dt \approx \int_0^T  \hat Q(T-t) (W^\top V)^{-1} W^\top V\hat F(t) dt = \int_0^T  \hat Q(T-t)\hat F(t) dt$ and $\int_0^T \tilde Q(T-t)^\top \left(\sum_{j=1}^{q} N_j k_{ij}\right) \tilde F(t) dt\approx \int_0^T \hat Q(T-t) (W^\top V)^{-1} W^\top\left(\sum_{j=1}^{q} N_j k_{ij}\right) V \hat F(t) dt =\int_0^T \hat Q(T-t) \left(\sum_{j=1}^{q}\hat N_j k_{ij} \right) \hat F(t) dt$ using the definition of $\hat N_j$ in \eqref{red_mat_projection}. Therefore, a small deviation is given in \eqref{optcond}.
\end{proof}
Conditions \eqref{almost_opt_cond} mean that the columns of the mixed (dual) covariance functions $\tilde F$ and $\tilde G$ have to take values close to the images of $V$ and $W$, respectively. However, these sufficient conditions for a small deviation can immediately be weakened for (a) and (b) in Theorem \ref{thm_opt_cond}. Assuming the reduced coefficients to be like in \eqref{red_mat_projection}, (a) and (b) become 
\begin{equation}\label{averge_for_aandb}\begin{aligned}
C \int_0^T V \hat F(t) - \tilde F(t) dt &=   CV\hat P(T) - C\tilde P(T) = 0,  \\
 X_0^\top \int_0^T W (V^\top W)^{-1} \hat G(t) - \tilde G(t) dt &= X_0^\top  W (V^\top W)^{-1}\hat Q(T) - X_0^\top  \tilde Q(T)=0                                                                                                                                               \end{aligned}
 \end{equation}
meaning that is enough to have a time average approximation in \eqref{almost_opt_cond} instead of having it point-wise. However, a good time average approximation, seems not to ensure the same in equations (c) and (d) of Theorem \ref{thm_opt_cond}. Therefore, the gap in \eqref{almost_opt_cond} shall be an indicator for the approximation error between \eqref{original_system} and \eqref{red_system}. By the error propagation in the reduction procedure, it is very likely that the deviation in \eqref{almost_opt_cond} is largest at time $T$. For that reason, the error is supposed to be checked there only. We will further emphasize the role of this terminal time covariance error in Section \ref{sec_part_method}.\smallskip

One might also think of constructing a reduced system  that exactly satisfies the optimality conditions of Theorem \ref{thm_opt_cond} but this is generally  a very hard task.

\subsection{Particular model reduction method}\label{sec_part_method}
For that reason, we continue with introducing an algorithm designed to meet the assumptions of Lemma \ref{lem_cond_opt}, i.e., we construct projection matrices $V$ and $W$ that represent bases for the dominant subspaces of the mixed covariances $\tilde F$ and $\tilde G$.  Such a scheme is presented in Algorithm \ref{algo:MBIRKA}. It is an iterative method since $\tilde F$ and $\tilde G$ themselves depend  on the matrices $V$ and $W$ that we desire to compute from these mixed covariances.
\begin{remark}
Setting $S=I$ in step \ref{step3} of Algorithm \ref{algo:MBIRKA}, we see that the equations in step \ref{step4} are the ones for $\tilde F$ and $\tilde G$. General state space transformations by $S$ basically do not change the dominant subspaces of these mixed covariances (up to some rescaling). Therefore, the solution spaces of $X$ and $Y$ can be considered instead. Now, as candidates for bases of the respective dominant subspaces (independent of $t$), $V$ and $W$ are derived from the time averages of $X$ and $Y$ in step \ref{step5}. We will discuss in Section \ref{sec_dom_sub} in which cases these candidates are a good choice. Methods similar to Algorithm \ref{algo:MBIRKA} exist in the control system context, where $S$ typically is the factor of the eigenvalue decomposition of $\hat A$, so that $\tilde A$ is a diagonal matrix of eigenvalues. Such schemes are so-called Iterative Rational Krylov Algorithms  \cite{breiten_benner, linIRKA, mliopt}. In the numerical Section \ref{sec_num}, $S=I$ will be our choice for pricing Bermudan options in a reduced framework based on Algorithm \ref{algo:MBIRKA}.
\end{remark}
The role of the deviation in \eqref{almost_opt_cond} at time $T$ and the time averaged covariance error has been discussed around \eqref{averge_for_aandb}. 
In the particular case of Algorithm \ref{algo:MBIRKA}, we show that these two factors have an impact on each other.

\begin{algorithm}[!tb]
	\caption{Sylvester fixed point iteration}
	\label{algo:MBIRKA}
	\begin{algorithmic}[1]
		\Statex {\bf Input:} The system matrices: $A, X_0, C, N_i$. Covariance matrix: $K_M = (k_{ij})$.
		\Statex {\bf Output:} The reduced matrices: $\hat A, \hat X_0, \hat C, \hat N_i$.
		\State Make an initial guess for the reduced matrices $\hat A, \hat X_0,\hat C, \hat N_i$.
\While {not converged}
		\State \label{step3} Perform state space transformation with regular matrix $S = S(\hat A, \hat X_0, \hat C, \hat N_i)$:
		\Statex\quad\qquad $\tilde A = S\hat A S^{-1},~\tilde X_0 = S\hat X_0, ~\tilde C = \hat C S^{-1}, ~\tilde N_i = S\hat N_i S^{-1}.$
		\State Compute the averages $\int_0^T X(t) dt$ and $\int_0^T Y(t) dt$ of
		
		\Statex\label{step4} \quad\qquad$\dot X(t) =  A X(t) + X(t) \tilde A^\top  + \sum_{i, j=1}^{q} N_i X(t) \tilde N_j^\top k_{ij}, \quad X(0)  = X_0\tilde X_0^\top$,
		\Statex \quad\qquad$\dot Y(t) = A^\top Y(t)  +  Y(t) \tilde A  + \sum_{i, j=1}^{q} N_i^\top Y(t) \tilde N_j k_{ij}, \quad Y(0) = C^\top\tilde C$.
\State \label{step5} $V = \orth{\left(\int_0^T X(t) dt\right)}$ and $W = \orth{\left(\int_0^T Y(t) dt\right)}$, where $\orth{(\cdot)}$ returns an orthonormal basis for the image of a matrix.
		\State Determine the reduced matrices:
		\Statex \quad\qquad$\hat A = (W^\top V)^{-1}W^\top AV,\quad \hat X_0 = (W^\top V)^{-1}W^\top X_0,\quad\hat C = CV$, \quad $\hat N_i = (W^\top V)^{-1}W^\top N_i V$.
		\EndWhile
	\end{algorithmic}
\end{algorithm}
\begin{thm}\label{thm:error_opt_cond}
Let $\hat A$, $\hat N_i$, $\hat X_0$ and $\hat C$ be the reduced-order matrices computed by Algorithm~\ref{algo:MBIRKA} assuming that it converged. Then, we have \begin{align*}
&(i)\quad V\hat P(T) - \tilde P(T) = V\; \hat {\mathcal L}^{-1}\Big[\hat F(T) - (W^\top V)^{-1} W^\top \tilde F(T)\Big],  \\
&(ii)\quad W (V^\top W)^{-1}\hat Q(T) - \tilde Q(T) = W (V^\top W)^{-1}\; \hat {\mathcal L}^{-*}\Big[\hat G(T) - V^\top \tilde G(T)\Big]                                                                                                                                                                  \end{align*}
meaning that \begin{align}\label{small_average_error}
V\hat F \approx \tilde F\quad\text{and}\quad W (V^\top W)^{-1} \hat G \approx \tilde G
             \end{align}
on average w.r.t. the probability measure $\frac{dt}{T}$ on $[0, T]$ given that  
 \begin{align}\label{terminal_time_cov_error}
V\hat F(T) \approx \tilde F(T)\quad\text{and}\quad
W (V^\top W)^{-1} \hat G(T)\approx \tilde G(T).
             \end{align}
        
\end{thm}
\begin{proof}
Below, we assume that $\hat A$, $\hat N_i$, $\hat X_0$, $\hat C$ is a fixed point of  Algorithm~\ref{algo:MBIRKA}. By step \ref{step5}, there exist regular matrices $M_V$ and $M_W$ such that \begin{align}\label{rep_VW}
V = \int_0^T X(t) dt\;M_V\quad \text{and}\quad W = \int_0^T Y(t) dt\;M_W.                                                                                                                                                                                                                                                                                                                                                                                     \end{align}
We multiply \eqref{matrixequalforFmixed} with $S^\top$ (see step \ref{step3}) from the right-hand side and obtain \begin{align*}
\frac{d}{dt} [\tilde F(t) S^\top] = A[\tilde F(t)S^\top] + [\tilde F(t)S^\top]\underbrace{S^{-\top}\hat A^\top S^\top}_{=\tilde A^\top} + \sum_{i, j=1}^{q} N_i [\tilde F(t)S^\top] \underbrace{S^{-\top} \hat N_j^\top S^\top}_{=\tilde N_j^\top} k_{ij},
\end{align*}
with $\tilde F(0) S^\top=X_0 \tilde X_0^\top$. Therefore, we have $\tilde F(t) = X(t) S^{-\top}$ and hence \begin{align}\label{rep_tildeP}
\tilde P(T) =  \int_0^T X(t) dt\; S^{-\top} = V M_V^{-1} S^{-\top}                                                                                                                                                                                                               \end{align}
inserting \eqref{rep_VW} above. On the other hand, integrating the equation for $X(t)$ in step \ref{step4} and using the representation of its integral in \eqref{rep_VW} yields \begin{align}\label{first_eq}
 X(T) -  X_0\tilde X_0^\top  &=  A [V M_V^{-1}] + [V M_V^{-1}] \tilde A^\top  + \sum_{i, j=1}^{q} N_i [V M_V^{-1}]\tilde N_j^\top k_{ij}\\ \label{sec_eq}
 \Leftrightarrow \tilde F(T) -  X_0\hat X_0^\top  &=  A V [M_V^{-1} S^{-\top}] + V [ M_V^{-1}S^{-\top} ] \hat A^\top  + \sum_{i, j=1}^{q} N_i V [M_V^{-1} S^{-\top}]\hat N_j^\top k_{ij}.                                                                                                                                                                                                                                                                                                                                              \end{align}
The equivalent formulation \eqref{sec_eq} above is obtained by multiplying \eqref{first_eq} with $S^{-\top}$ from the right and by exploiting the definitions of $\tilde A$, $\tilde X_0$, $\tilde N_j$ as well as by using  $\tilde F(T) = X(T) S^{-\top}$. We multiply \eqref{sec_eq} with $(W^\top V)^{-1} W^\top$ from the left resulting in \begin{align}\label{bla1}
 (W^\top V)^{-1} W^\top \tilde F(T) -  \hat X_0\hat X_0^\top  =  \hat {\mathcal L}[M_V^{-1} S^{-\top}].                                                                                                                                                                                                                                                                                                                                                                                                                                                                                                                                                                                                                  \end{align}
We further know that \begin{align}\label{bla2}
\hat F(T) -  \hat X_0\hat X_0^\top  =  \hat {\mathcal L}[\hat P(T)]                                                                                                                                                                                                                                                                                                                                                                                                                                                                                                                                                                                                               \end{align}
by integrating \eqref{matrixequalforFred}. Subtracting \eqref{bla1} from \eqref{bla2} and applying the inverse of $\hat {\mathcal L}$, we find $\hat P(T) - M_V^{-1} S^{-\top} =  \hat {\mathcal L}^{-1}\Big[\hat F(T) - (W^\top V)^{-1} W^\top \tilde F(T)\Big]$. We multiply this equation with $V$ from the left such that (i) follows by \eqref{rep_tildeP}. \\
Let us now multiply \eqref{matrixequalforGmixed} with $S^{-1}$ from the right leading to \begin{align*}
\frac{d}{dt} [\tilde G(t) S^{-1}] = A^\top [\tilde G(t)S^{-1}] + [\tilde G(t)S^{-1}]\underbrace{S\hat A S^{-1}}_{=\tilde A} + \sum_{i, j=1}^{q} N_i^\top [\tilde G(t) S^{-1}] \underbrace{S \hat N_j S^{-1}}_{=\tilde N_j} k_{ij},
\end{align*}
with $\tilde G(0) S^{-1}=C^\top \tilde C$. This yields $\tilde G(t) = Y(t) S$. Therefore, using  \eqref{rep_VW}, we obtain \begin{align}\label{rep_tildeQ}
\tilde Q(T) =  \int_0^T Y(t) dt\; S = W M_W^{-1} S.                                                                                                                                                                                                               \end{align}
We integrate the second differential equation in step \ref{step4} and insert \eqref{rep_VW} into the resulting identity. This provides \begin{align}\label{first_eqQ}
 Y(T) -  C^\top \tilde C  =  A^\top [W M_W^{-1}] + [W M_W^{-1}] \tilde A  + \sum_{i, j=1}^{q} N_i^\top [W M_W^{-1}]\tilde N_j k_{ij}\end{align}
Now, multiplying \eqref{first_eqQ} with $S$ from the right, this is equivalent to \begin{align}
\label{sec_eqQ}
\tilde G(T) -  C^\top \hat C  =  A^\top W [M_W^{-1} S] + W [ M_W^{-1}S] \hat A  + \sum_{i, j=1}^{q} N_i^\top W [M_W^{-1} S]\hat N_j k_{ij}.                                                                                                                                                                                                                                                                                                                                              \end{align}
Multiplying \eqref{sec_eqQ} with $V^\top$ from the left and including the identity matrix $(V^\top W)^{-1} V^\top W $, we have  
\begin{align}\nonumber
V^\top\tilde G(T) -  \hat C^\top \hat C  &=  \hat A^\top [V^\top W M_W^{-1} S] + [V^\top W  M_W^{-1}S] \hat A  + \sum_{i, j=1}^{q} \hat N_i^\top [V^\top W M_W^{-1} S]\hat N_j k_{ij}\\ \label{somelabel}
&= \hat{\mathcal L}^*[V^\top W M_W^{-1} S].                                                                                                                                                                                                                                                                                                                                              \end{align}
By integrating \eqref{matrixequalforGred}, we obtain $\hat  G(T) -  \hat C^\top \hat C  = \hat{\mathcal L}^*[\hat Q(T)]$. Subtracting \eqref{somelabel} from this equation and applying $\hat{\mathcal L}^{-*}$ yields $\hat Q(T) - V^\top W M_W^{-1} S = \hat{\mathcal L}^{-*}[\hat G(T) - V^\top\tilde G(T)]$. Taking \eqref{rep_tildeQ} into account relation (ii) follows by multiplication of $W (V^\top W)^{-1}$ from the left.
\end{proof}
\begin{remark}\label{remark2}
Our strategy to derive a suitable reduced order model is to compute a fixed point of Algorithm \ref{algo:MBIRKA} for which \eqref{terminal_time_cov_error} holds true. As a consequence of Theorem \ref{thm:error_opt_cond}, we obtain \eqref{small_average_error} which determines the error in (a) and (b) of Theorem \ref{thm_opt_cond} according to \eqref{averge_for_aandb}. With \eqref{small_average_error}  and taking into account that (by the error propagation) the point-wise covariance error might be largest at $T$, we have strong indicators for meeting the assumptions of Lemma \ref{lem_cond_opt} giving us a small error in conditions (c) and (d) of Theorem \ref{thm_opt_cond}. Therefore, we expect Algorithm \ref{algo:MBIRKA} to provide reduced matrices satisfying the optimality conditions in \eqref{optcond} up to a small deviation in case the covariance error is small at $T$. This makes the associated reduced system a candidate for a good approximation of the original large-scale model. Notice that the errors in \eqref{terminal_time_cov_error} can be computed exactly since $\hat F, \tilde F, \hat G, \tilde G$ can be calculated from the vectorized versions of \eqref{matrixequalforFred}, \eqref{matrixequalforFmixed},  \eqref{matrixequalforGred}, \eqref{matrixequalforGmixed}. This requires to have $\expn^{\hat{\mathcal K} T}$, $\hat{\mathcal K}\in\mathbb R^{\hat n^2\times \hat n^2}$,  and $\expn^{\tilde{\mathcal K} T}$,  $\tilde{\mathcal K}\in\mathbb R^{(\hat n \cdot n)\times (\hat n\cdot n)}$, which are easily available as long as $\hat n\cdot n$ is not too large.
\end{remark}
\subsection{Dominant subspaces and further intuition behind the choice of $V$ and $W$}\label{sec_dom_sub}
At this point, it is still not fully clear whether it is realistic for \eqref{terminal_time_cov_error} to hold when Algorithm \ref{algo:MBIRKA} is applied, i.e., we need to discuss under which circumstances $\im[V]$ and $\im[W]$ are suitable subspaces approximating the solution spaces of $\tilde F$ and $\tilde G$. For that reason, 
we proceed with identifying when it can be expected that \eqref{terminal_time_cov_error} or, more generally, the assumptions of Lemma \ref{lem_cond_opt} are satisfied. In this context, we also answer the question concerning suitable reduced dimensions $\hat n$ ensuring a high accuracy in the approximation.
\paragraph{Solution spaces of $\tilde F$ and $\tilde G$.}
The following lemma illustrates the solution spaces for $\tilde F$ and $\tilde G$ which are the basis for detecting the dominant subspaces of these mixed covariance functions.
\begin{lem}\label{solution_space}
The columns of $\tilde F(t), \tilde P(T)$, $t\in [0, T]$ are contained in $\im[P(T)]$ and the columns of $\tilde G(t), \tilde Q(T)$, $t\in [0, T]$, lie in $\im[Q(T)]$.
\end{lem}
\begin{proof}
Suppose that $z_T\in \kernel[P(T)]$. Then, we have \begin{align}\label{Fequalszero}                                                                                                                                                                                                                                                        \int_0^T z_T^\top F(t) z_T \;dt = z_T^\top P(T) z_T = 0.                                                                                                                                                                                                                                                                                                                                                                                                                                                                                                                                                                                                                                     \end{align}
By the stochastic representation $F(t)= \mathbb E \left[\Phi(t) X_0 X_0^\top \Phi(t)^\top\right]$, we know that $z_T^\top F(t) z_T \geq 0$. In addition, $t \mapsto F(t)$ is continuous because it is the solution to 
\eqref{matrixdgl}. Therefore, \eqref{Fequalszero} implies $0= z_T^\top F(t) z_T = \mathbb E \left[z_T^\top \Phi(t) X_0 X_0^\top \Phi(t)^\top z_T\right] = \mathbb E \|\left(\Phi(t) X_0\right)^\top z_T \|_2^2$ for all $t\in [0, T]$. This yields \begin{align}\label{fund_orth}                                                                                                                                                                                                                                        z_T^\top \Phi(t) X_0 = 0, \quad \mathbb P-\text{a.s. and for all } t\in [0, T].                                                                                                                                                                                                                                            \end{align}
As a consequence of \eqref{fund_orth}, we obtain $z_T^\top \tilde F(t) = \mathbb E \left[z_T^\top \Phi(t) X_0 \hat X_0^\top \hat \Phi(t)^\top\right] = 0$ for all $t$ and hence also $z_T^\top \tilde P(T) = \int_0^T z_T^\top \tilde F(t) dt = 0$. This means that the columns of $\tilde F(t)$ and $\tilde P(T)$ are orthogonal to $\kernel[P(T)]$. Since $P(T)$ is symmetric, we have $\im[P(T)]=\left(\kernel[P(T)]\right)^\perp$ such that the first part of the claim follows. The second part of the proof is omitted since it uses completely analogue arguments.
\end{proof}
\paragraph{Dominant subspaces of $\tilde F$, $\tilde G$ and algebraic criteria for a low-dimensional underlying structure.}
We now identify the spaces in which $\tilde F$ and $\tilde G$ can be well approximated. We consider the following diagonalizations of $P(T)$ and $Q(T)$: 
\begin{subequations}\label{balancing_PQ}
\begin{align}\label{diagP}
P(T) &= \mathcal T_P\Sigma_P(T)\mathcal T_P^\top,\\ \label{diagQ}
Q(T) &= \mathcal T_Q \Sigma_Q(T) \mathcal T_Q^{\top},
\end{align}   
\end{subequations}
where $\Sigma_\ell(T)  =\diag(\sigma_{\ell, 1}(T),\ldots,\sigma_{\ell, n}(T))$ is a diagonal and $\mathcal T_\ell$ is a regular matrix for $\ell\in\{P, Q\}$. $\Sigma_\ell(T)$ can be the matrix of eigenvalues of $P$ and $Q$, respectively. Alternatively, we can find a transformation with $\Sigma_P(T)= \Sigma_Q(T)$. Given $P(T), Q(T)>0$, there exists a regular matrix $\mathcal T$ such that $ \mathcal T_P=\mathcal T^{-1}$, 
$\mathcal T_Q=\mathcal T^{\top}$ and $\Sigma_\ell(T)  =\diag(\sigma_{1}(T),\ldots,\sigma_{n}(T))$ (for both $\ell\in\{P, Q\}$), see \cite{mor_heston} for more details, where the Hankel singular values (HSVs) $\sigma_{i}(T)$ are the square roots of the eigenvalues of $P(T) Q(T)$. The latter transformation is called balancing and can be seen as a simultaneous diagonalization of $P(T)$ and $Q(T)$. This is the basis for the MOR scheme in \cite{mor_heston} that is related to balanced truncation, a MOR technique that was introduced for deterministic linear control systems in \cite{moore}.\smallskip

The decompositions in \eqref{balancing_PQ} provide a weighted basis of $\im[P]$ and $\im[Q]$ represented by  $\mathcal T_P$ and $\mathcal T_Q$, respectively. By Lemma \ref{solution_space}, the dominant subspace of $\tilde F$ and $\tilde G$ can therefore be identified which are the eigenspaces corresponding to the large eigenvalues of $P(T), Q(T)$ or the subspaces associated to the large HSVs. This will be further emphasized in Theorem \ref{thm_rel_irka_BT} below. In this context,
let us partition \begin{align}\label{defineVW}
	\mathcal T_Q= \begin{pmatrix}
			W_{Q} &  \mathcal T_{Q, 2}
		\end{pmatrix}, \quad
\mathcal T_P= \begin{pmatrix}
			V_{P} &  \mathcal T_{P, 2} 
		\end{pmatrix}, \quad\Sigma_\ell(T) = \begin{pmatrix}
        \Sigma_{\ell, 1}(T)& \\
        & \Sigma_{\ell, 2}(T)
       \end{pmatrix},                                                                                                                                                        \end{align}
where $\Sigma_{\ell, 2}(T)=\diag(\sigma_{\ell, \hat n+1}(T),\ldots,\sigma_{\ell, n}(T))$, $\ell\in\{P, Q\}$, is the matrix of small eigenvalues/HSVs and $V_{P}$, $W_{Q}\in \mathbb R^{n\times \hat n}$. In the following theorem, it can be seen that the images of $V$ and $W$ resulting from Algorithm \ref{algo:MBIRKA} are a good approximation for the dominant subspaces of $\tilde F$ and $\tilde G$.
\begin{thm}\label{thm_rel_irka_BT}
Let $V, W$ be the matrices obtained by Algorithm \ref{algo:MBIRKA} given that it converged and let $V_{P}, W_{Q}$ be defined as in \eqref{defineVW}. Then, there exist $M_{\tilde F}(t)\in \mathbb R^{\hat n\times\hat n}, M_{\tilde F, 2}(t)\in \mathbb R^{(n-\hat n)\times\hat n}$ ($t\in [0, T]$) and $M_{\tilde P}(T) \in \mathbb R^{\hat n\times\hat n}, M_{\tilde P, 2}(T)\in \mathbb R^{(n-\hat n)\times\hat n}$ such that \begin{equation}\label{relV}
\begin{aligned}
 \tilde F(t) &=      V_{P} M_{\tilde F}(t) +  \mathcal T_{P, 2} \Sigma_{P, 2}(T) M_{\tilde F, 2}(t), \\ V &=      V_{P} M_{\tilde P}(T) +  \mathcal T_{P, 2} \Sigma_{P, 2}(T) M_{\tilde P, 2}(T).
 \end{aligned}
 \end{equation}
 Moreover, we find $M_{\tilde G}(t)\in \mathbb R^{\hat n\times\hat n}, M_{\tilde G, 2}(t)\in \mathbb R^{(n-\hat n)\times\hat n}$ and $M_{\tilde Q}(T)\in \mathbb R^{\hat n\times\hat n}, M_{\tilde Q, 2}(T)\in \mathbb R^{(n-\hat n)\times\hat n}$ giving us \begin{equation}\label{relW}\begin{aligned}
 \tilde G(t) &=      W_{Q} M_{\tilde G}(t) +  \mathcal T_{Q, 2} \Sigma_{Q, 2}(T) M_{\tilde G, 2}(t), \\ W &=      W_{Q} M_{\tilde Q}(T) +  \mathcal T_{Q, 2} \Sigma_{Q, 2}(T) M_{\tilde Q, 2}(T).
 \end{aligned}
 \end{equation}
\end{thm}
\begin{proof}
By Lemma \ref{solution_space}, there are matrices $Z_{\tilde F}(t), Z_{\tilde P}(T)\in \mathbb R^{n\times\hat n}$ such that $\tilde F(t) = P(T) Z_{\tilde F}(t)$ and $\tilde P(T) = P(T) Z_{\tilde P}(T)$. Exploiting the representation of $P(T)$ in  \eqref{diagP} and the partition of $\mathcal T_P$ in \eqref{defineVW} then leads to \begin{align*}                                                                                                                                                                                                \tilde F(t) =  \mathcal T_P \Sigma_P(T) \mathcal T_P^{\top} Z_{\tilde F}(t) =  V_{P} \Sigma_{P, 1}(T) V_{P}^{\top} Z_{\tilde F}(t) +  \mathcal T_{P, 2} \Sigma_{P, 2}(T) \mathcal T_{P, 2}^{\top} Z_{\tilde F}(t).                                                                                                                                                                                                    \end{align*}
The same way, we obtain
\begin{align*}                                                                                                                                                                                                         V= \tilde P(T) S^{\top} M_V =  V_{P} \Sigma_{P, 1}(T) V_{P}^{\top} Z_{\tilde P}(T)S^{\top} M_V +  \mathcal T_{P, 2} \Sigma_{P, 2}(T) \mathcal T_{P, 2}^{\top} Z_{\tilde P}(T) S^{\top} M_V
\end{align*}
based on \eqref{rep_tildeP}. This provides \eqref{relV}. In order to show \eqref{relW}, we exploit Lemma \ref{solution_space}. Therefore, we have $\tilde G(t) = Q(T) Z_{\tilde G}(t)$ and $\tilde Q(T) = Q(T) Z_{\tilde Q}(T)$ for some $Z_{\tilde G}(t), Z_{\tilde Q}(T)\in \mathbb R^{n\times\hat n}$. We insert \eqref{diagQ} into these equations and partition $\mathcal T_Q$ as in \eqref{defineVW} yielding \begin{align*}                                                                                                                                                                                                \tilde G(t) =  \mathcal T_Q \Sigma_Q(T) \mathcal T_Q^\top Z_{\tilde G}(t) =  W_{Q} \Sigma_{Q, 1}(T) W_{Q}^{\top} Z_{\tilde G}(t) +  \mathcal T_{Q, 2} \Sigma_{Q, 2}(T) \mathcal T_{Q, 2}^{\top} Z_{\tilde G}(t).                                                                                                                                                                                                   \end{align*}
Similarly, using \eqref{rep_tildeQ}, it holds that \begin{align*}                                                                                                                                                                                                         W= \tilde Q(T) S^{-1} M_W =  W_{Q} \Sigma_{Q, 1}(T) W_{Q}^{\top} Z_{\tilde Q}(T)S^{-1} M_W +  \mathcal T_{Q, 2} \Sigma_{Q, 2}(T) \mathcal T_{Q, 2}^{\top} Z_{\tilde Q}(T) S^{-1} M_W.
\end{align*}
This concludes the proof.
\end{proof}
Theorem \ref{thm_rel_irka_BT} tells us that the eigenvalues of $P(T), Q(T)$ or the HSVs $\sigma_{1}(T),\ldots,\sigma_{n}(T)$ determine whether the dominant subspaces of the system and its dual are low-dimensional. If we find a small $\hat n$ such that the diagonal entries of $\Sigma_{P, 2}(T)$ (matrix of truncated eigenvalues of $P(T)$ or HSVs) are small, then \eqref{relV} shows that $\tilde F$ approximately takes values in the low-dimensional subspace $\im[V_{P}]$. On the other hand, \eqref{relV} also tells that the columns of $V$ represent an orthonormal basis  that, up to small perturbations depending on $\Sigma_{P, 2}(T)$, lives in $\im[V_{P}]$. For that reason, we know that the columns of $V$ approximately are a basis for $\im[V_{P}]$ ($\im[V_{P}]$ and $\im[V]$ are close in some sense). Relations \eqref{relW} give us the same information about $\tilde G$. If the neglected eigenvalues of $Q(T)$ or the HSVs $\sigma_{\hat n+1}(T),\ldots,\sigma_{n}(T)$ are small for a small $\hat n$, we can approximate $\tilde G$ well in $\im[W_{Q}]$ which itself can be accurately represented by the low-dimensional space $\im[W]$.\smallskip

Summing up these arguments, given a low-dimensional underlying structure of the problem (characterized by a large number of small eigenvalues/HSVs), we find good approximations of $\tilde F$ and $\tilde G$ in low-dimensional subspaces $\im[V]$ and $\im[W]$, respectively. Knowing that $V \hat F$ and $W(V^\top W)^{-1}\hat G$ are projection based estimates of $\tilde F$ and $\tilde G$, respectively, in these spaces, we can expect \eqref{terminal_time_cov_error} and the assumptions of Lemma \ref{lem_cond_opt} to hold in case $\sigma_{\ell, \hat n+1}(T),\ldots,\sigma_{\ell, n}(T)\approx 0 $ for $\ell\in\{P, Q\}$ or $\sigma_{\hat n+1}(T),\ldots,\sigma_{n}(T)\approx 0$.
\begin{remark}
 With the truncated eigenvalues of $P(T)$ and $Q(T)$ or the associated HSVs strong algebraic criteria for the error in the optimality conditions in Theorem \ref{thm_opt_cond} are found. However, this requires to solve for the integrals of the solutions of the potentially very high-dimensional matrix differential equations \eqref{matrixdgl} and \eqref{matrixequalforG}. 
 This is can be very expensive or infeasible for large $n$. Therefore, the criterion of checking for the error in \eqref{almost_opt_cond} at time $T$, see Remark \ref{remark2}, is more accessible as it is related to the lower dimensional equations \eqref{matrixequalforFred}, \eqref{matrixequalforFmixed}, \eqref{matrixequalforGred} and \eqref{matrixequalforGmixed}.
\end{remark}

\section{Reduced order modelling given asymptotic stability}\label{mor_stab_sys}

Asset price models often involve dividends $\delta>0$ such as in the classical example of Andersen and Broadie \cite{andersonbroadie}, where \eqref{stochstate} is given in the component-wise form $dx_i(t) =(\mathbf{\mathrm{r}} - \delta) x_i(t)dt+ \xi_i x_i(t) dM_i(t)$, $i\in\{1, \dots, n\}$. Like in \cite{andersonbroadie}, we have $\mathbb E\left\|x_i(t)\right\|_2^2\rightarrow 0$, as $t\rightarrow \infty$, regardless of the initial state given that $\delta$ is sufficiently large, i.e., it holds that $2 (\mathbf{\mathrm{r}} - \delta)+ \xi_i^2 k_{ii}<0$ for all $i\in\{1, \dots, n\}$.
In this section, we discuss a modification of Algorithm \ref{algo:MBIRKA} if such a stability condition holds true. In fact, we present the asymptotic behavior of Algorithm \ref{algo:MBIRKA} in terms of the optimality conditions \eqref{optcond} for the bound in Lemma \ref{lemma_error_measure} as $T\rightarrow \infty$. In order to ensure existence of the limit in \eqref{error_measure_bound} as $T\rightarrow \infty$, the above mentioned mean square asymptotic stability of \eqref{stochstate} is required which is generally defined by the property that 
 $\mathbb E\left\|\Phi(t)\right\|^2 \lesssim \expn^{-c t}$ for some constant $c>0$. We refer to \cite{staboriginal} or \cite{redmannspa2} for more details on this condition and associated equivalent algebraic formulations. With this assumption, the covariance function $F$ decays to zero exponentially as $T\rightarrow \infty$. If \eqref{red_stochstate} is mean square stable as well, the same holds true for the reduced and the mixed covariance functions $\hat F$ and $\tilde F$. Therefore, $P:=\lim_{T\rightarrow \infty} P(T)$, $\hat P:=\lim_{T\rightarrow \infty} \hat P(T)$ and $\tilde P:=\lim_{T\rightarrow \infty} \tilde P(T)$ exist in \eqref{error_measure_bound}. In addition, mean square asymptotic stability of \eqref{stochstate} is equivalent to mean square asymptotic stability of the dual equation \eqref{stochstate_dual}. This also guarantees the existence of 
 $Q:=\lim_{T\rightarrow \infty} Q(T)$, $\hat Q:=\lim_{T\rightarrow \infty} \hat Q(T)$ and $\tilde Q:=\lim_{T\rightarrow \infty} \tilde Q(T)$. Let us recall that the dual integrated covariance functions $Q(T), \hat Q(T), \tilde Q(T)$ were introduced at the end of Section \ref{blabla}. Now, the right-hand side of  \eqref{error_measure_bound} can be replaced by its limit which represents a more conservative bound, i.e., we obtain  
 \begin{align}\label{Limit_error_measure_bound}
\mathbb E\int_0^T \left\|y(t)-\hat y(t) \right\|_2^2dt \leq \left(\trace(C P C^\top)-2 \trace(C \tilde P \hat C^\top)+\trace(\hat C \hat P \hat C^\top) \right)\left\|z_0\right\|_2^2. \end{align}
Necessary optimality conditions for this limit bound can be proved much easier than for \eqref{error_measure_bound} due to several vanishing terms but, in fact,  such conditions for an expression similar to the right-hand side of \eqref{Limit_error_measure_bound} have already been obtained in the context of stochastic control system \cite{mliopt}. We take this result and formulate it in a theorem below. Moreover, we prove that these optimality conditions are consistent with the ones of Theorem \ref{thm_opt_cond}.
\begin{thm}\label{limit_thm_opt_cond}
Given that \eqref{original_system} and the reduced model \eqref{red_system} are  mean square asymptotically stable and suppose that the matrices $\hat A, \hat N_i, \hat X_0, \hat C$ of reduced system are locally optimal with respect to the bound in \eqref{Limit_error_measure_bound}. Then, it holds that 
\begin{equation}\label{optcond_limit}
\begin{aligned}
           (a)\quad&\hat C \hat P = C \tilde P,\quad (b)\quad \hat Q\hat X_0  =  \tilde Q^\top X_0, \quad
           (c)\quad  \hat Q\hat P  =\tilde Q^\top \tilde P, \\
           (d)\quad&\hat Q \left(\sum_{j=1}^{q}\hat N_j k_{ij} \right) \hat P =   \tilde Q^\top \left(\sum_{j=1}^{q} N_j k_{ij}\right) \tilde P
\end{aligned}
\end{equation}
for $i=1,\dots, q$. Moreover, taking the limit as $T\rightarrow \infty$ in \eqref{optcond} results in \eqref{optcond_limit}.
\end{thm}
\begin{proof}
The conditions in \eqref{optcond_limit} directly follow from \cite[Theorem 2.4]{mliopt}, where stochastic control systems are investigated. Given mean square asymptotic stability, taking the limit in (a) and (b) in \eqref{optcond}, we immediately obtain relations (a) and (b) in this theorem. For (c) and (d), different representations for $\tilde Q(T)$ and $\hat Q(T)$ need to be found. First of all, by \eqref{matrixequalforGred}, we have $\vect(\hat G(t)) = \expn^{\hat {\mathcal K}^\top t} \vect(\hat C^\top \hat C)$. Integrating this equation over $[0, T]$ yields $\vect(\hat Q(T)) = \big(\expn^{\hat {\mathcal K}^\top T}- I\big)\hat {\mathcal K}^{-\top} \vect(\hat C^\top \hat C)$. By the exponential decay of $\hat G$, it holds that $\expn^{\hat {\mathcal K}^\top T}\rightarrow 0$ for $T\rightarrow \infty$. Therefore, we have  $\vect(\hat Q) = - \hat {\mathcal K}^{-\top} \vect(\hat C^\top\hat C)$, so that 
$\vect(\hat Q(T)) = \big(I-\expn^{\hat {\mathcal K}^\top T}\big)\vect(\hat Q)$. Devectorizing this equation leads to $\hat Q(T) = \hat Q - \hat G(T, \hat Q)$, where $\hat G(\cdot, \hat Q)$ is the solution to \eqref{matrixequalforGred} with initial state $\hat Q$. Exploiting 
\eqref{matrixequalforGmixed} and using the exactly same steps, we obtain  $\tilde Q(T) = \tilde Q - \tilde G(T, \tilde Q)$ with $\tilde G(\cdot, \tilde Q)$ solving \eqref{matrixequalforGmixed} having the initial value $\tilde Q$. Using these new representations of $\hat Q(T)$ and $\tilde Q(T)$, (c) in \eqref{optcond} becomes
\begin{align}\nonumber
&\int_0^T   \bigg(\hat Q - \hat G(T-t, \hat Q)\bigg)\hat F(t) dt  =\int_0^T  \bigg(\tilde Q - \tilde G(T-t, \tilde Q)\bigg)^\top \tilde F(t) dt\\ \label{cond_c_equiv}
& \Leftrightarrow
\hat Q \hat P(T)-\int_0^T \hat G(T-t, \hat Q)\hat F(t) dt  =\tilde Q \tilde P(T)- \int_0^T  \tilde G(T-t, \tilde Q)^\top \tilde F(t) dt.
\end{align}
By the mean square asymptotic stability of the of the (dual) full and the (dual) reduced system, we find a constant $c>0$, so that $\left\|\tilde G(T-t, \tilde  Q)\right\|, \left\|\hat G(T-t, \hat Q)\right\|\lesssim \expn^{-c(T-t)}$ and $\left\|\tilde F(t)\right\|, \left\|\hat F(t)\right\|\lesssim \expn^{-ct}$. Consequently, taking the limit of $T\rightarrow \infty$ in \eqref{cond_c_equiv}, we obtain (c) in \eqref{optcond_limit}. The arguments for (d) are analogue which concludes the proof. 
\end{proof}
The conditions in \eqref{optcond_limit} are closely related to optimality conditions for certain error measures corresponding to deterministic bilinear control systems \cite{morZhaL02}. In particular, the results of \cite{morZhaL02} can be viewed as a special case of Theorem \ref{limit_thm_opt_cond}. We now prove that a modification of Algorithm \ref{algo:MBIRKA} satisfies the optimality conditions in \eqref{optcond_limit}. This can actually be done by the techniques used in 
\cite{breiten_benner}(bilinear control systems) or in \cite{mliopt} (stochastic control systems). However, we simply exploit the steps of the proof of Theorem \ref{thm:error_opt_cond} since they immediately yield the following result.
\begin{thm}\label{somethm}
Suppose that the inputs of Algorithm \ref{algo:MBIRKA} are asymptotically mean square stable systems \eqref{original_system} and \eqref{red_system}. Then, replacing $T$ by $\infty$ in steps \ref{step4} and \ref{step5} of Algorithm \ref{algo:MBIRKA}  and assuming that this algorithm converges, the matrices  $\hat A$, $\hat N_i$, $\hat X_0$ and $\hat C$ define a reduced system satisfying \begin{align}\label{perfect_average_fit}
 V\hat P   = \tilde P\quad \text{and} \quad W (V^\top W)^{-1}\hat Q = \tilde Q.                                                                                                                                                                 \end{align}
As a consequence, we obtain that the coefficients of this reduced system fulfill the optimality conditions in Theorem \ref{limit_thm_opt_cond}.
\end{thm}
\begin{proof}
Based on the underlying stability, we can set $T=\infty$ in the proof of Theorem \ref{thm:error_opt_cond}. Using the same notation like in that proof, we then obtain $\tilde P = V M_V^{-1} S^{-\top}$ as in \eqref{rep_tildeP}. Since all covariance functions tend to zero as $T\rightarrow \infty$, \eqref{bla1} and \eqref{bla2} become $
 -  \hat X_0\hat X_0^\top  =  \hat {\mathcal L}[M_V^{-1} S^{-\top}]$ and
$-\hat X_0\hat X_0^\top  =  \hat {\mathcal L}[\hat P] $, respectively. Consequently, we have $\hat P = M_V^{-1} S^{-\top}$ and hence $V \hat P = \tilde P$. Moreover, we have $
\tilde Q = W M_W^{-1} S$ from \eqref{rep_tildeQ} again by setting $T=\infty$ in each step before.
The disappearing covariances lead to
$-  \hat C^\top \hat C  =  \hat{\mathcal L}^*[V^\top W M_W^{-1} S]$ instead of \eqref{somelabel} which is also the identify for $\hat Q$. Consequently, we obtain  $\hat Q=V^\top W M_W^{-1} S$. Multiplying this equation with $W (V^\top W)^{-1}$ from the left, we get  $W (V^\top W)^{-1}\hat Q = \tilde Q$ yielding the first part of the claim.  Now, using \eqref{perfect_average_fit}, we obtain $C \tilde P =  C V\hat P = \hat C \hat P$, $\tilde Q^\top X_0=\hat Q(W^\top V)^{-1} W^\top X_0 = \hat Q\hat X_0$, $\tilde Q^\top \tilde P= \hat Q(W^\top V)^{-1} W^\top V\hat P= \hat Q \hat P$ and $\tilde Q^\top \left(\sum_{j=1}^{q} N_j k_{ij}\right) \tilde P = \hat Q(W^\top V)^{-1} W^\top\left(\sum_{j=1}^{q} N_j k_{ij}\right) V \hat P = \hat Q\left(\sum_{j=1}^{q} \hat N_j k_{ij}\right)\hat P$. These are the identities in \eqref{optcond_limit} closing the proof.                                                                                                                                                                                                                                                                                                                        
\end{proof}
The above result might also not be too surprising following the discussion of Section \ref{section3}, where we noticed that the covariance error \eqref{terminal_time_cov_error} at $T$ determines the error in the optimality conditions. In this section, $T$ is  basically replaced by $\infty$, where the covariance error vanishes due to the underlying stability leading to a perfect fit of the optimality conditions. 
\begin{remark}
The change in Algorithm \ref{algo:MBIRKA} proposed in Theorem \ref{somethm} generally leads a worse approximation because a larger bound \eqref{Limit_error_measure_bound} is minimized in comparison to \eqref{error_measure_bound}. However, computational complexity is lower since the required $\mathcal X_\infty:=\int_0^\infty X(t) dt$ and $\mathcal Y_\infty:=\int_0^\infty Y(t) dt$ in step \ref{step4} of the modified scheme are solutions to \begin{align*}                                                                                                                                                                                                                                                    
-X_0\tilde X_0^\top &=  A \mathcal X_\infty + \mathcal X_\infty \tilde A^\top  + \sum_{i, j=1}^{q} N_i \mathcal X_\infty \tilde N_j^\top k_{ij},\\
-C^\top\tilde C&= A^\top \mathcal Y_\infty  +  \mathcal Y_\infty \tilde A  + \sum_{i, j=1}^{q} N_i^\top \mathcal Y_\infty \tilde N_j k_{ij}.                                                                                                                                                                                                                                                                                                                                                                                                                                                      \end{align*}
These matrix equations can be solved if $\hat n \cdot n$ is very large, see, e.g., \cite{damm_matrixeq}. This is in contrast to the original Algorithm \ref{algo:MBIRKA}, where the left-hand sides of the above equations additionally involve $X(T)$ and $Y(T)$, respectively. These terminal values need to be calculated from vectorizations like in \eqref{matrixequalforFmixed} and \eqref{matrixequalforGmixed}. This is not feasible given that $\hat n \cdot n>10^4$.
\end{remark}
In the following section, the benefit of our dimension reduction methods in the context of pricing Bermudan options is presented.

\section{Applying Algorithm \ref{algo:MBIRKA} to price Bermudan options}  \label{sec_num}
In this section, we make use of the dimension reduction schemes, investigated in this paper, in the context of option pricing. In particular, we construct a reduced asset price model by Algorithm \ref{algo:MBIRKA} with $S=I$ that (approximately) satisfies the optimality conditions of Theorems \ref{thm_opt_cond} with the expectation of having the reduced system close to the original model in the $L^2$-sense using inequality \eqref{error_measure_bound}. Based on this $L^2$-approximation, we aim for  \eqref{approx_stop} giving an accurate estimate for a Bermudan option price.

\subsection{Basket call options}\label{sec_basket} 

Suppose that $M=B$ is a Wiener process with covariance matrix $K_B$.
For $i= 1, \ldots, 50$, we study the following Black-Scholes model 
\begin{align}\label{scholesnum2}
             dx_i(t) =(\mathbf{\mathrm{r}} - \delta) x_i(t)dt+ \xi_i x_i(t) dB_i(t),\quad x_i(0)=x_{0, i},\quad t\in [0, T],
            \end{align}
with interest rate $\mathbf{\mathrm{r}} = 0.02$, dividends $\delta = 0.07$ and terminal time $T=1$. Moreover, we set $X_0=x_0$ and $ z_0=1$. The volatilities $\xi_i$ are sampled from a uniform distribution on $[0.1, 0.3]$ and we pick the initial states $x_{0, i}$ randomly from the interval $[0.1, 1.4]$. 
For this example, we assume that each component $B_i$ of $B$ is a standard Wiener process and additionally the associated covariance matrix $K_B$ contains both small and large correlations. We can see from Figure \ref{valuesKw} that the entries of $K_B$ are between $0.1$ and $1$ and that the ratio of large and small correlations is roughly the same. \smallskip

The quantity of interest is a basket with equal weights, i.e., \begin{align}   \label{basket}                                                        y(t) = \sum_{i=1}^{50} x_i(t).                                                                            \end{align}
Consequently, the matrix $C$ is a row vector of ones. By Lemma \ref{lem_cond_opt}, we know that our scheme needs to find matrices $V$ and $W$ such that $\im[V]$ and $\im[W]$ represent the dominant subspaces of the mixed covariance functions $\tilde F$ and $\tilde G$ in order to have a small error in \eqref{optcond}. If the truncated HSVs $\sigma_{\hat n+1}(T), \dots, \sigma_n(T)$ ($\sigma_i(T)$ are the square roots of the eigenvalues of $P(T) Q(T)$) are small, such a scheme is given by Algorithm \ref{algo:MBIRKA}, see Theorem \ref{thm_rel_irka_BT}. These values are depicted in Figure \ref{hsvplot}. There, we can see that there is only a single large HSV dominating the other ones indicating that a scalar reduced system can capture large parts of the dynamics of original asset price model. Based on the HSVs, we further expect a very good approximation by \eqref{red_system} choosing $\hat n\geq 2$. Whether  Algorithm \ref{algo:MBIRKA} is almost optimal in terms of conditions  \eqref{optcond} can also be checked by the corresponding covariance error \eqref{almost_opt_cond} at the terminal time $T$, see also Remark \ref{remark2}.  This error is displayed in Table \ref{table0}. We observe that the dominant subspaces of $\tilde F(T)$ and $\tilde G(T)$ are captured well by the projection matrices $V$ and $W$ indicating a small error in \eqref{optcond}. The covariance errors in Table \ref{table0}  indeed give a good intuition for the order of the $L^2$-error. Looking at the first column of Table \ref{table1}, a small relative $L^2$-error is given for all reduced order dimensions considered there. We also see that the error is decreasing in $\hat n$. The second column Table \ref{table1} provides the error of applying the model reduction scheme proposed in \cite{mor_heston} to our problem. We observe that both schemes perform equally well. However, our method has the advantage of being  generally applicable in higher dimensions as it is computationally less expensive than the scheme used in \cite{mor_heston}. Let us further point out that the norm of the basket here is  $\left\|y\right\|_{L^2_T}:= \sqrt{\int_0^T \mathbb E \left\|y(t)\right\|^2 dt} \approx 28$ such that the absolute $L^2$-error is by that factor larger than the relative one.
\begin{figure}[ht]
\begin{minipage}{0.48\linewidth}
 \hspace{-0.5cm}
\includegraphics[width=1.1\textwidth,height=6cm]{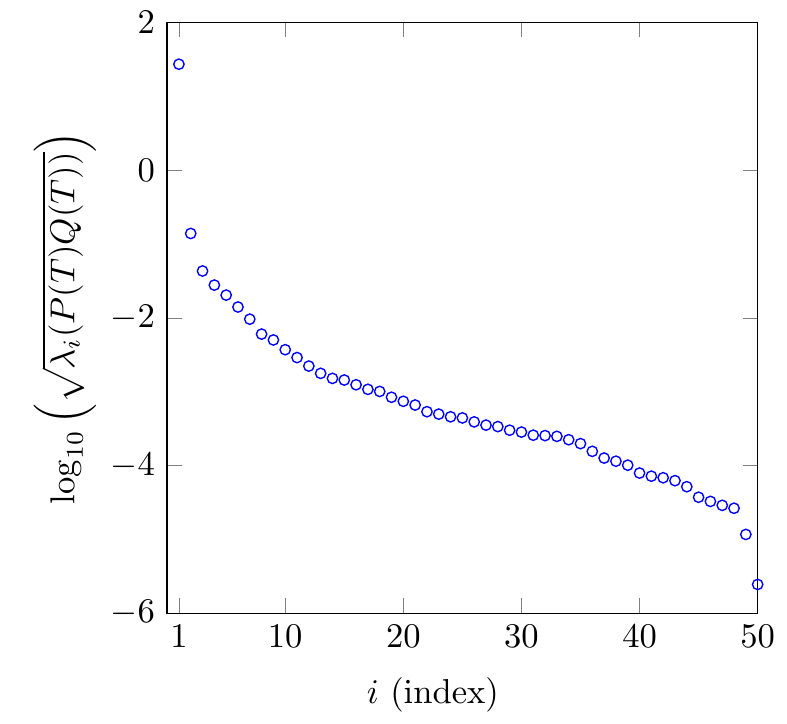}
\caption{Logarithmic HSVs of \eqref{scholesnum2} with associated basket in \eqref{basket}.}\label{hsvplot}
\end{minipage}\hspace{0.5cm}
\begin{minipage}{0.48\linewidth}
 \hspace{-0.5cm}
\includegraphics[width=1.1\textwidth,height=6cm]{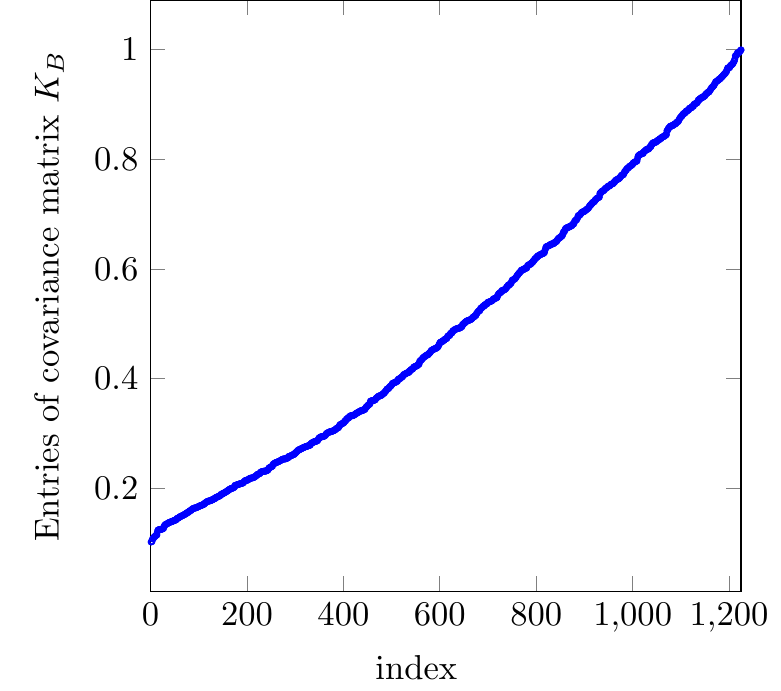}
\caption{Ordered entries $k_{ij}$, $i<j$, of the covariance matrix $K_B$ for the basket call option.}\label{valuesKw}
\end{minipage}
\end{figure}
\begin{table}
\begin{minipage}{.98\linewidth}
\centering\begin{tabular}{|c|c|c|}\hline
& \multicolumn{2}{c}{Algorithm \ref{algo:MBIRKA}} \vline \\
\hline 
$\hat n$  & $\textcolor{white}{\Big\|}\|V\hat F(T) - \tilde F(T)\|/\|\tilde F(T)\|$ & $\|W (V^\top W)^{-1} \hat G(T) - \tilde G(T)\|/\|\tilde G(T)\|$\\
\hline
\hline
$1$ & $2.62$e$-03$& $2.94$e$-03$ \\   
$2$ &$2.60$e$-03$  &$2.92$e$-03$ \\
$3$ &$2.60$e$-03$  & $2.93$e$-03$ \\
$4$ &$1.53$e$-03$ & $1.82$e$-03$\\
$5$ &$6.83$e$-04$ &$8.72$e$-04$  \\
\hline
\end{tabular}\caption{The (dual) covariance error \eqref{terminal_time_cov_error} of Algorithm \ref{algo:MBIRKA} with $S=I$ at time $T$.}
\label{table0}
\end{minipage}
\begin{minipage}{.98\linewidth}
\centering\begin{tabular}{|c|c|c|}\hline
& \multicolumn{2}{c}{$\left\|y-\hat y\right\|_{L^2_T} /\left\|y\right\|_{L^2_T}$} \vline \\
\hline 
$\hat n$  & Algorithm \ref{algo:MBIRKA} & Method from \cite{mor_heston}\\
\hline
\hline
$1$ & $4.52$e$-03$& $4.51$e$-03$\\   
$2$ &$1.75$e$-03$  &$1.75$e$-03$ \\
$3$ &$1.19$e$-03$  &$1.18$e$-03$\\
$4$ &$8.43$e$-04$ & $8.45$e$-04$\\
$5$ &$5.88$e$-04$ &$5.89$e$-04$\\
\hline
\end{tabular}\caption{Relative $L^2$- error between the basket \eqref{basket} and the output $\hat y$ associated to the reduced system based on Algorithm \ref{algo:MBIRKA} with $S=I$ and the method studied in \cite{mor_heston}.}
\label{table1}
\end{minipage}
\end{table}
\smallskip

Below, we consider a Bermudan option associated to the asset price model \eqref{scholesnum2} with basket \eqref{basket}. In this context, we introduce the discounted payoff function \begin{align*}
f_t(y(t)) = \expn^{-\mathbf{\mathrm{r}} t} \max\left\{y(t)-\kappa, 0\right\}                 
                \end{align*}
with strike price $\kappa = y(0)$ describing the gain of the basket option when exercising it at time $t$. Further,  we assume to have a discrete set of dates $\mathcal{J} = \left\{0, 0.25, 0.5, 0.75, 1\right\}$, where the option can be exercised.
%
As sketched in Section \ref{intro2}, we want to approximate $u$ by $\hat u$ in this context, where
\begin{equation}
  \label{eq:optimal-stopping-value}
  u := \sup_{\tau \in \mathcal{S}_0} \mathbb{E}[f_\tau(y(\tau))]\quad\text{and}\quad \hat{u} := \sup_{\tau \in \hat{\mathcal{S}}_0} \mathbb{E}[f_\tau(\hat{y}(\tau))]
\end{equation}
with $\mathcal{S}_0$ and $\hat{\mathcal{S}}_0$ denoting the set of all stopping times w.r.t.~the
filtration generated by the price processes $x$ and $\hat x$, respectively. Since $\hat u$ is the optimal value of a stopping problem that is different from the original one, $\hat u$ generally is no lower bound of $u$. However, we find a bound for the error between $u$ and $\hat u$ if we formally replace $\mathcal{S}_0$ and $\hat{\mathcal{S}}_0$ by the class of stopping times w.r.t.~the
filtration generated by the noise process $B$ in \eqref{eq:optimal-stopping-value} taking into account that there is no gain of the enlargement of $\mathcal{S}_0$ and $\hat{\mathcal{S}}_0$. Now, having defined both problems in \eqref{eq:optimal-stopping-value} on the same set of stopping times, we can use the result in \cite{mor_heston}. Therefore, we have  \begin{equation}\label{eq:reduced-pathwise-L1}
    \vert u - \hat{u}\vert \leq \mathbb{E}\Big[ \sup_{t \in \mathcal{J}} \left\vert f(y(t)) - f(\hat{y}(t))\right\vert
    \Big].
  \end{equation}
In the following,  the fair price $\hat u$ of the Bermudan option in the reduced setting is computed exploiting  the algorithm of Longstaff and Schwartz \cite{LS2001}. This method computes the continuation functions based on regression, i.e., it involves solving problems of the form \eqref{eq:least-squares-regression}. In this context, the basis functions $\psi_{k}$ ($k=1\dots, \mathfrak K-1$) are chosen to be polynomials of absolute order of at most $4$. Moreover, the payoff function is added to the basis as well. Therefore, the total number of basis elements is $\mathfrak K = \frac{(4+\hat n)!}{4! \,\hat n!}+1$. Additionally, we use $\mathfrak{M} = 10^6$ samples within the regression procedure. Tables \ref{table_bermudan} and \ref{table_bermudan3} show the Bermudan option prices in the reduced setting with $\hat n = 1, 2, 3, 4, 5$ using Algorithm \ref{algo:MBIRKA} or the method of \cite{mor_heston}, respectively. As for the $L^2$-error, the performance is pretty good for both schemes. Based on the corresponding error bound, we know that the price for $\hat n=5$ is very close to the actual one. On the other hand, comparing the results for the dimensions $\hat n = 1$ and $\hat n=5$, the gain is very little taking into account that the standard deviation of the estimator is of order $10^{-3}$.
\begin{table}[th]
\begin{minipage}{.95\linewidth}
\centering
\begin{tabular}{|c|c|c|}\hline
$\hat n$  & Value $\hat u$ of Bermudan option reduced system & $\mathbb{E}\left[ \sup_{t \in \mathcal{J}} \vert f(y(t)) - f(\hat{y}(t))\vert\right]$\\
\hline
\hline
$1$ & $0.99157$& $0.090439$\\
$2$ & $0.99356$ & $0.036384$\\
$3$ & $0.99436$ & $0.024370$\\
$4$ & $0.99358$ & $0.017471$ \\
$5$ & $0.99494$ & $0.012379$ \\
\hline
\end{tabular}\caption{Fair price of the Bermudan basket call option in reduced system based on Algorithm \ref{algo:MBIRKA} with $S=I$ and $K_B$ with entries as in Figure \ref{valuesKw}.}
\label{table_bermudan}
\end{minipage}
\begin{minipage}{.95\linewidth}
\centering
\begin{tabular}{|c|c|c|}\hline
$\hat n$  & Value $\hat u$ of Bermudan option reduced system & $\mathbb{E}\left[ \sup_{t \in \mathcal{J}} \vert f(y(t)) - f(\hat{y}(t))\vert\right]$\\
\hline
\hline
$1$ & $0.99193$& $0.090443$\\
$2$ & $0.99469$ & $0.036351$\\
$3$ & $0.99328$ & $0.024341$\\
$4$ & $0.99433$ & $0.017504$ \\
$5$ & $0.99500$ & $0.012365$ \\
\hline
\end{tabular}\caption{Fair price of the Bermudan basket call option in the reduced system based on the method studied in \cite{mor_heston} and $K_B$ with entries as in Figure \ref{valuesKw}.}
\label{table_bermudan3}
\end{minipage}
\end{table}

\subsection{Max call Bermudan options}

Asset price models with  baskets \eqref{basket} can be significantly reduced as such baskets represent low-dimensional partial information of the underlying high-dimensional problem. Generally, if the number $p$ of outputs (dimension of $y$) is small and if there are many high correlations between the different noise processes, then the reduction potential is very high. This is an observation that was also made in \cite{Bermudan_POD,mor_heston}, where dimension reduction techniques have been applied to Bermudan options. At least one of these essential factors has to be satisfied in order to be able to achieve an accurate reduced system, in which regression is feasible. If a large $p$ is combined with uncorrelated noise, then we might still be able to lower the complexity but $\hat n\leq 10$ will barely be possible (of course this also depends on the original $n$).\smallskip

Now, when aiming to approximate the value of a max call option, the quantity of interest \begin{align*}
   \mathbf {y_{max}}(t) = \max\{x_1(t), x_2(t),\dots, x_n(t)\}                                                                                                                                                                               \end{align*}
is still a functional of the vector of asset prices. However, it is a highly nonlinear one not immediately fitting the setting in \eqref{original_system}. To solve this problem, we set $y(t) = x(t)$ (or $C=I$) with the goal of approximating the full state by $\hat y(t) = V\hat x(t)$. Given a high accuracy, we consequently have \begin{align*}
   \mathbf {y_{max}}(t) \approx \mathbf {\hat y_{max}}(t):=  \max\{\hat y_1(t), \hat y_2(t),\dots, \hat y_n(t)\}                                                                                                                                                                                                                                                                                                                                                                                                                                                                                                                                                                                                      \end{align*}
with $\hat y_i$ being the components of the above $\hat y$. However, this approach involves a large number of outputs as we have $p=n$. Therefore, the same reduction performance like in the basket case of Section \ref{sec_basket} can only be achieved changing the noise profile. In particular, we add more high correlations. The entries of the new $K_B$ are depicted in Figure \ref{valuesKw_max}. Moreover, in order to roughly have the same option price as in the previous section, we also modify the initial states. We pick them from a uniform distribution on the interval $[5, 6]$. The Bermudan option prices for the reduced model using Algorithm \ref{algo:MBIRKA} are stated in Table \ref{table_bermudan2}. According to the error bounds the approximation is very accurate for $\hat n=6$. In the full state approximation case, we see a significant difference between the chosen reduced dimensions $\hat n$ in terms of the performance. In fact, $\hat n=1, 2$ yield a rather poor estimate. This can be explained by the high order of the associated truncated HSVs depicted in Table \ref{hsvplot_max}. Since the fourth HSV is relatively small, an acceptable performance can be seen from reduced systems with $\hat n\geq 3$. A good approximation is   achieved by fixing $\hat n\geq 5$ since the HSVs decay rapidly at this point.
\begin{figure}[ht]
\begin{minipage}{0.48\linewidth}
 \hspace{-0.5cm}
\includegraphics[width=1.1\textwidth,height=6cm]{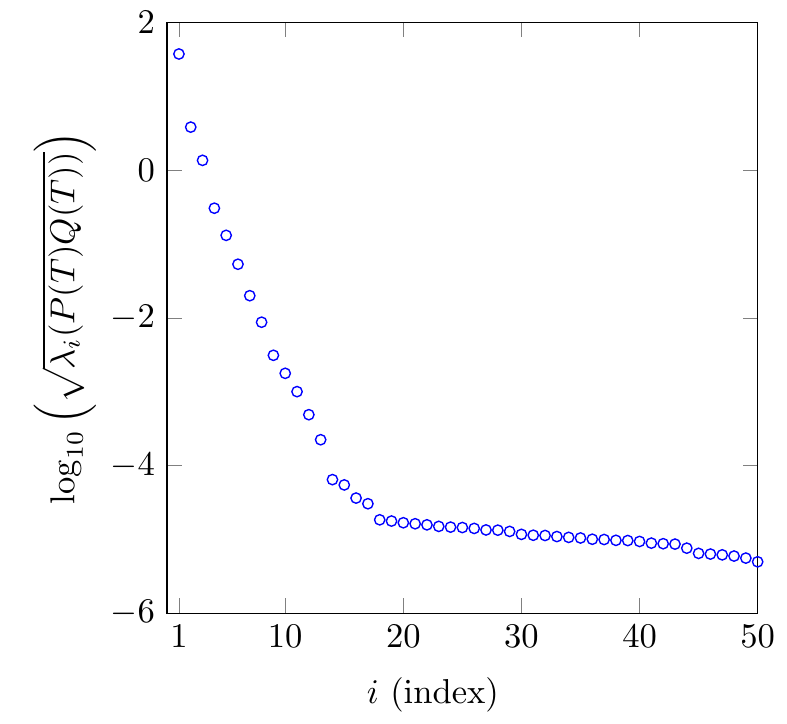}
\caption{Logarithmic HSVs of \eqref{scholesnum2} approximating the full state $x$ ($C=I$).}\label{hsvplot_max}
\end{minipage}\hspace{0.5cm}
\begin{minipage}{0.48\linewidth}
 \hspace{-0.5cm}
\includegraphics[width=1.1\textwidth,height=6cm]{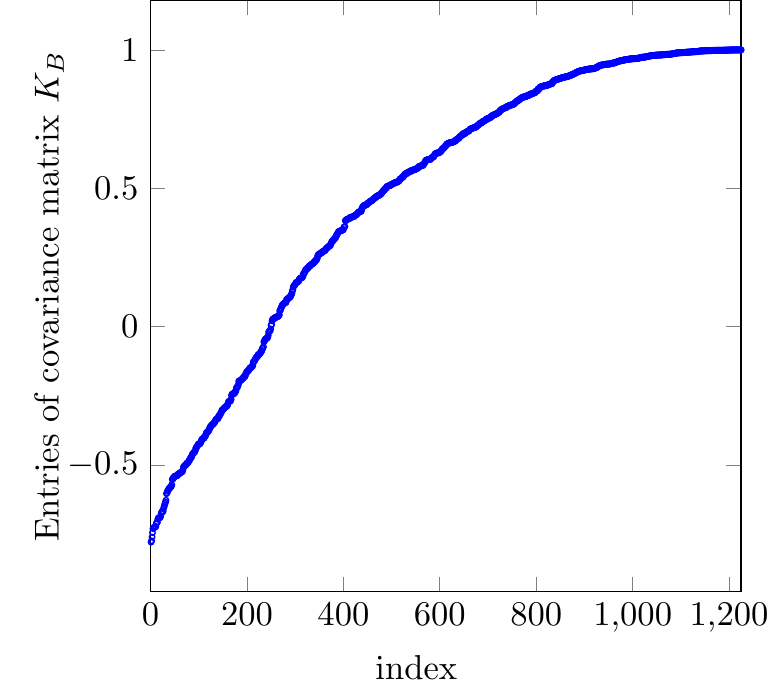}
\caption{Ordered entries $k_{ij}$, $i<j$, of the covariance matrix $K_B$ for the max call option.}\label{valuesKw_max}
\end{minipage}
\end{figure}
\begin{table}
\centering
\begin{tabular}{|c|c|c|}\hline
$\hat n$  & Value $\hat u$ of Bermudan option reduced system & $\mathbb{E}\left[ \sup_{t \in \mathcal{J}} \vert f(\mathbf {y_{max}}(t)) - f(\mathbf {\hat y_{max}}(t))\vert\right]$\\
\hline
\hline
$1$ & $0.20255$& $0.98777$\\
$2$ & $0.72525$ & $0.28064$\\
$3$ & $0.83729$ & $0.099617$\\
$4$ & $0.86855$ & $0.039208$ \\
$5$ & $0.88562$ & $0.015911$ \\
$6$ & $0.89057$ & $0.0050141$ \\
\hline
\end{tabular}\caption{Fair price of the Bermudan max call option in the reduced system based on Algorithm \ref{algo:MBIRKA} with $C=I$, $S=I$ and $K_B$ having entries as in Figure \ref{valuesKw_max}.}
\label{table_bermudan2}
\end{table}

\appendix

\section{Proof of Theorem \ref{thm_opt_cond}}\label{app_pending_proof}
Below, the pending proof of Theorem \ref{thm_opt_cond} follows:
\begin{proof}
If the coefficients of system \eqref{red_system} are locally minimal  with respect to the right-hand side of \eqref{error_measure_bound}, then we have $\partial_z \mathcal E = 0 $ with $\mathcal{E}$ being defined in \eqref{calE}. This is true if and only if, we have
\begin{align}\label{ifoptthen}
 \partial_z\langle {\hat C}^\top \hat C, \hat P(T)\rangle_F =\partial_z 2 \langle  C^\top {\hat C}, \tilde P(T)\rangle_F,                                   
\end{align}
where $z\in  \{\hat a_{k\ell}, \hat n^{(l)}_{k\ell}, \hat x_{k\ell}, \hat c_{k \ell}\}$ represents the entries of $\hat A = (\hat a_{k\ell})$, $\hat N_l = (\hat n^{(l)}_{k\ell})$, $\hat X_0 = (\hat x_{k\ell})$, $\hat C = (\hat c_{k \ell})$, $l=1, \dots, q$ and $k, \ell$ are generic indices. In the following, 
$e_k$ denotes the $k$th unit vector of suitable dimension. Setting $z = \hat c_{k\ell}$, relation \eqref{ifoptthen} becomes 
\begin{align*}
&\langle e_\ell e_{k}^\top\hat C + \hat C^\top   e_k e_\ell^\top, \hat P(T)\rangle_F =2 \langle  C^\top e_k e_\ell^\top, \tilde P(T)\rangle_F\\
&\Leftrightarrow \trace([e_\ell e_{k}^\top\hat C + \hat C^\top   e_k e_\ell^\top] \hat P(T))=  2 \trace(e_\ell e_{k}^\top C \tilde P(T) ).                                                                                                    
\end{align*}
Exploiting the properties of the trace and the symmetry of $\hat P$, this is equivalent to 
\begin{equation*}
e_{k}^\top{\hat C} \hat P(T) e_\ell=e_{k}^\top C \tilde P(T) e_\ell
\end{equation*}
for all $k= 1,\ldots p$ and $\ell= 1, \ldots, \hat n$. This yields identity (a).
Let us now focus on the cases $z = \hat a_{k\ell}, \hat n^{(l)}_{k\ell}, \hat x_{k\ell}$. Let us first introduce 
 $\tilde Q$ and $\hat Q$ as the solutions to \begin{align}\label{infinite_gram}
\hat {\mathcal L}^*[\hat Q]  =  - \hat C^\top \hat C, \quad                                                                                                                                                                                                                                                                                                                                            
  \tilde {\mathcal L}^*[\tilde Q] =-C^\top \hat C.                                                                                                                                                                                                                                                                                                                                                 \end{align}
Like in \eqref{matrixequalforGred} and \eqref{matrixequalforGmixed}, we can exploit that $\hat {\mathcal L}^*$ and $\tilde {\mathcal L}^*$ have the Kronecker matrix representations $\hat {\mathcal K}^\top$ and $\tilde {\mathcal K}^\top$, so that \eqref{infinite_gram} is equivalent to the vectorized version $\hat {\mathcal K}^\top \vect[\hat Q]  =  - \vect[\hat C^\top \hat C]$ and $                                                                                                                                                                                                                                                                                                                                           
  \tilde {\mathcal K}^\top\vect[\tilde Q] =-\vect[C^\top \hat C]$. 
Now, \eqref{ifoptthen} can be rewritten in the following equivalent ways
\begin{align}\nonumber
&\left\langle {\hat C}^\top \hat C, \partial_z \hat P(T)\right\rangle_F = 2 \left\langle  C^\top {\hat C}, \partial_z \tilde P(T)\right\rangle_F\Leftrightarrow \left\langle \hat {\mathcal L}^*[\hat Q], \partial_z \hat P(T) \right\rangle_F = 2 \left\langle  \tilde {\mathcal L}^*[\tilde Q], \partial_z \tilde P(T)\right\rangle_F\\
&\Leftrightarrow \left\langle \hat Q, \hat {\mathcal L}[\partial_z \hat P(T)] \right\rangle_F = 2 \left\langle  \tilde Q, \tilde {\mathcal L}[\partial_z \tilde P(T)]\right\rangle_F \label{frobenius_formulation}
\end{align}
using \eqref{infinite_gram}and the properties of adjoint operators. We integrate \eqref{matrixequalforFred} and \eqref{matrixequalforFmixed} with respect to time and apply $\partial_z$ to both sides of the resulting equations. This yields 
\begin{align}\label{rel1}
&\partial_z[\hat F(t)] - \partial_z[\hat X_0 \hat X_0^\top] =\partial_z \left[\hat {\mathcal L}[\hat P(t)]\right] = \hat {\mathcal L}[\partial_z \hat P(t)] + (\partial_z \hat {\mathcal L})[\hat P(t)],\\ \label{rel2}
&\partial_z[\tilde F(t)] - \partial_z[X_0 \hat X_0^\top] =\partial_z \left[\tilde {\mathcal L}[\tilde P(t)]\right] = \tilde {\mathcal L}[\partial_z \tilde P(t)] + (\partial_z \tilde {\mathcal L})[\tilde P(t)]
    \end{align}
for $t\in [0, T]$ using the product rule. We insert \eqref{rel1} and \eqref{rel2} for $t=T$ 
into \eqref{frobenius_formulation} and obtain
\begin{align}\label{rel3}
 \left\langle \hat Q, \partial_z[\hat F(T)] - \partial_z[\hat X_0 \hat X_0^\top]- (\partial_z \hat {\mathcal L})[\hat P(T)] \right\rangle_F = 2 \left\langle  \tilde Q, \partial_z[\tilde F(T)] - \partial_z[X_0 \hat X_0^\top]-(\partial_z \tilde {\mathcal L})[\tilde P(T)]\right\rangle_F.
\end{align}
Below, we determine a representation for $\partial_z[\hat F(t)]$ and $\partial_z[\tilde F(t)]$. Therefore, we apply the time derivative to both \eqref{rel1} and \eqref{rel2} providing
\begin{equation}\label{rel4}
\begin{aligned}
&\frac{d}{dt}\partial_z[\hat F(t)]  = \hat {\mathcal L}\Big[\partial_z [\hat F(t)]\Big] + (\partial_z \hat {\mathcal L})[\hat F(t)],\quad \partial_z[\hat F(0)] = \partial_z[\hat X_0 \hat X_0^\top], \\ 
&\frac{d}{dt}\partial_z[\tilde F(t)]  = \tilde {\mathcal L}\Big[\partial_z [\tilde F(t)]\Big] + (\partial_z \tilde {\mathcal L})[\tilde F(t)],\quad \partial_z[\tilde F(0)] = \partial_z[X_0 \hat X_0^\top].
    \end{aligned}
\end{equation}
We now fix $z=  \hat x_{k\ell}$ and observe that $(\partial_{\hat x_{k\ell}} \hat {\mathcal L}), (\partial_{\hat x_{k\ell}} \tilde {\mathcal L}) = 0$ since  $\hat {\mathcal L}$ and $\tilde {\mathcal L}$ do not depend on $\hat X_0$. In this case, the vectorized forms  of \eqref{rel4} are analogue to the ones of \eqref{matrixequalforFred} and \eqref{matrixequalforFmixed}. The solution representations of these vectorizations and the relation $\left\langle A_1, A_2 \right\rangle_F= \vect(A_1)^\top \vect(A_2)$ for two generic matrices $A_1, A_2$ of suitable dimension are exploited below. Consequently, \eqref{rel3} for $z=  \hat x_{k\ell}$ reads as follows \begin{align*}
 &\left\langle \hat Q, \partial_z[\hat F(T)] - \partial_z[\hat X_0 \hat X_0^\top] \right\rangle_F = 2 \left\langle  \tilde Q, \partial_z[\tilde F(T)] - \partial_z[X_0 \hat X_0^\top]\right\rangle_F\\
 &\Leftrightarrow (\vect(\hat Q))^\top \vect\left(\partial_z[\hat F(T)] - \partial_z[\hat X_0 \hat X_0^\top]\right)= 2 (\vect(\tilde Q))^\top \vect\left(\partial_z[\tilde F(T)] - \partial_z[ X_0 \hat X_0^\top]\right)\\
 &\Leftrightarrow (\vect(\hat Q))^\top \left(\expn^{\hat{\mathcal K}T} - I\right)\vect\left(\partial_z[\hat X_0 \hat X_0^\top]\right)= 2 (\vect(\tilde Q))^\top \left(\expn^{\tilde{\mathcal K}T} - I\right)\vect\left(\partial_z[X_0 \hat X_0^\top]\right)\\
 &\Leftrightarrow \Big(\hat {\mathcal K}^{-\top} \vect[\hat C^\top \hat C] \Big)^\top \left(\expn^{\hat{\mathcal K}T} - I\right)\vect\left(\partial_z[\hat X_0 \hat X_0^\top]\right)
 = 2 \Big(\tilde {\mathcal K}^{-\top} \vect[C^\top \hat C]\Big)^\top \left(\expn^{\tilde{\mathcal K}T} - I\right)\vect\left(\partial_z[X_0 \hat X_0^\top]\right)\\
 &\Leftrightarrow \left((\expn^{\hat{\mathcal K}^\top T} - I) \hat{\mathcal K}^{-\top}\vect(\hat C^\top \hat C)\right)^\top \vect\left(\partial_z[\hat X_0 \hat X_0^\top]\right) = 2 \left((\expn^{\tilde{\mathcal K}^\top T} - I) \tilde{\mathcal K}^{-\top}\vect(C^\top \hat C)\right)^\top \vect\left(\partial_z[X_0 \hat X_0^\top]\right)\\
 &\Leftrightarrow \left(\int_0^T\expn^{\hat{\mathcal K}^\top t} \vect(\hat C^\top \hat C) dt\right)^\top \vect\left(\partial_z[\hat X_0 \hat X_0^\top]\right)= 2 \left(\int_0^T\expn^{\tilde{\mathcal K}^\top t} \vect(C^\top \hat C) dt\right)^\top  \vect\left(\partial_z[X_0 \hat X_0^\top]\right)\\
 &\Leftrightarrow \left(\vect(\hat Q(T))\right)^\top \vect\left(\partial_z[\hat X_0 \hat X_0^\top]\right)= 2 \left(\vect(\tilde Q(T))\right)^\top  \vect\left(\partial_z[X_0 \hat X_0^\top]\right)\\
 &\Leftrightarrow\left\langle \hat Q(T), \partial_z[\hat X_0 \hat X_0^\top] \right\rangle_F = 2 \left\langle  \tilde Q(T), \partial_z[X_0 \hat X_0^\top]\right\rangle_F \\
 &\Leftrightarrow\left\langle \hat Q(T), e_k e_\ell^\top \hat X_0^\top + \hat X_0 e_\ell e_k^\top \right\rangle_F = 2 \left\langle  \tilde Q(T), X_0 e_\ell e_k^\top \right\rangle_F
\end{align*}
inserting the vectorizations of \eqref{infinite_gram} and the definitions of $\hat Q(T)$, $\tilde Q(T)$. The definition of the Frobenius inner product and properties of the trace now lead to $e_k^\top Q(T) \hat X_0 e_\ell = e_k^\top\tilde Q(T)^\top X_0 e_\ell$ for all $k=1, \dots, \hat n$ and $\ell = 1, \dots, m$. This gives us condition (b).
    
We proceed with $z\in  \{\hat a_{k\ell}, \hat n^{(l)}_{k\ell}\}$ and use that in this case, we have $\partial_z[X_0 \hat X_0^\top], \partial_z[\hat X_0 \hat X_0^\top]=0$ in \eqref{rel3} and \eqref{rel4}. For that reason, 
the vectorization of the first equation in \eqref{rel4} has the (mild) solution representation $\vect(\partial_z[\hat F(t)])= \int_0^t \expn^{\hat{\mathcal K} (t - s)} \vect\left( (\partial_z \hat {\mathcal L})[\hat F(s)]\right) ds$. Therefore, we find the following representation for the left-hand side in \eqref{rel3} based on $\left(I-\expn^{\hat{\mathcal K}^\top T}\right)\vect(\hat Q) = \vect(\hat Q(T))$:
\begin{align*}
 \left\langle \hat Q, \partial_z[\hat F(T)] - (\partial_z \hat {\mathcal L})[\hat P(T)] \right\rangle_F &= (\vect(\hat Q))^\top \vect\left(\partial_z[\hat F(T)]-(\partial_z \hat {\mathcal L})[\hat P(T)\right) \\
 &= \vect(\hat Q)^\top \int_0^T \left(\expn^{\hat{\mathcal K} (T - t)} - I\right) \vect\left( (\partial_z \hat {\mathcal L})[\hat F(t)]\right) dt\\
 &= -\int_0^T \vect(\hat Q(T-t))^\top\vect\left( (\partial_z \hat {\mathcal L})[\hat F(t)]\right) dt \\
 &= -\int_0^T \left\langle \hat Q(T-t), (\partial_z \hat {\mathcal L})[\hat F(t)]\right\rangle_F dt
\end{align*}
exploiting that $(\partial_z \hat {\mathcal L})[\hat P(T)]=\int_0^T (\partial_z \hat {\mathcal L})[\hat F(t)] dt$. With the same steps, we also obtain for the right-hand side of \eqref{rel3} that\begin{align*}
 \left\langle \tilde Q, \partial_z[\tilde F(T)] - (\partial_z \tilde {\mathcal L})[\tilde P(T)] \right\rangle_F = -\int_0^T \left\langle \tilde Q(T-t), (\partial_z \tilde {\mathcal L})[\tilde F(t)]\right\rangle_F dt.
\end{align*}
For that reason, \eqref{rel3} is equivalent to \begin{align}\label{relan}
 \int_0^T \left\langle \hat Q(T-t), (\partial_z \hat {\mathcal L})[\hat F(t)]\right\rangle_F dt = 2 \int_0^T \left\langle \tilde Q(T-t), (\partial_z \tilde {\mathcal L})[\tilde F(t)]\right\rangle_F dt.
\end{align}
It remains to determine the partial derivatives of the Lyapunov operators. For $z=a_{k\ell}$, we find that 
$(\partial_z \hat {\mathcal L})[\hat X] = e_k e_\ell^\top \hat X^\top + \hat X e_\ell e_k^\top$ and $(\partial_z \tilde {\mathcal L})[\tilde X] = \tilde X e_\ell e_k^\top$. Applying this to \eqref{relan}, we have 
\begin{align*}
 \int_0^T \left\langle \hat Q(T-t), e_k e_\ell^\top\hat F(t) + \hat F(t) e_\ell e_k^\top\right\rangle_F dt = 2 \int_0^T \left\langle \tilde Q(T-t), \tilde F(t)e_\ell e_k^\top\right\rangle_F dt.
\end{align*}
By the definition of the Frobenius inner product and properties of the trace, we obtain \begin{align*}
 e_k^\top \int_0^T  \hat Q(T-t)\hat F(t) dt \;e_\ell =e_k^\top \int_0^T  \tilde Q(T-t)^\top \tilde F(t) dt\; e_\ell
\end{align*} 
for all $k, \ell = 1, \dots, \hat n$ providing (c). 
We now define $\hat\Psi_i^\top := \sum_{j=1}^{q} {\hat N}_{j}^\top k_{ij}$ and observe that 
$\partial_{\hat n^{(l)}_{k\ell}} \hat\Psi_i^\top = e_\ell e_k^\top k_{il}$. Consequently, we have 
\begin{align*}
 \partial_{\hat n^{(l)}_{k\ell}}\hat {\mathcal L} &= \partial_{\hat n^{(l)}_{k\ell}} \sum_{i, j=1}^{q} \hat N_i (\cdot) \hat N_j^\top k_{ij} =   \partial_{\hat n^{(l)}_{k\ell}} \sum_{i=1}^{q} \hat N_i (\cdot) \hat\Psi_i^\top = e_k e_\ell^\top (\cdot) \hat\Psi_l^\top + \sum_{i=1}^{q} \hat N_i (\cdot) e_\ell e_k^\top k_{il}  \\
 &=    e_k e_\ell^\top (\cdot) \hat\Psi_l^\top + \hat\Psi_l (\cdot)e_\ell e_k^\top, 
\end{align*}
since $k_{il}=k_{li}$. Analogue to the above steps, we obtain  
\begin{align*}
\partial_{\hat n^{(l)}_{k\ell}}\tilde {\mathcal L} = \partial_{\hat n^{(l)}_{k\ell}} \sum_{i, j=1}^{q} N_i (\cdot) \hat N_j^\top k_{ij}=  \partial_{\hat n^{(l)}_{k\ell}} \sum_{i=1}^{q} N_i (\cdot) \hat\Psi_i^\top = \sum_{i=1}^{q} N_i (\cdot)  e_\ell e_k^\top k_{il} = \Psi_l (\cdot)e_\ell e_k^\top, 
\end{align*}
where we set $\Psi_l := \sum_{j=1}^{q} {N}_{j} k_{lj}$. Consequently, for $z=\hat n^{(l)}_{k\ell}$, \eqref{relan} becomes \begin{align*}
 \int_0^T \left\langle \hat Q(T-t),   \hat e_k e_\ell^\top \hat F(t) \hat\Psi_l^\top + \hat\Psi_l \hat F(t)  e_\ell e_k^\top\right\rangle_F dt = 2 \int_0^T \left\langle \tilde Q(T-t), \Psi_l \tilde F(t) e_\ell e_k^\top\right\rangle_F dt
\end{align*}
for $l=1, \dots, q$. This is equivalent to \begin{align*}
 e_k^\top\int_0^T \hat Q(T-t) \hat\Psi_l \hat F(t) dt\;e_\ell =  e_k^\top\int_0^T \tilde Q(T-t)^\top \Psi_l \tilde F(t) dt\;e_\ell 
\end{align*}
for all $k, \ell = 1, \dots, \hat n$ concluding the proof.
\end{proof}

\bibliographystyle{plain}

\end{document}